\documentclass[12pt,fleqn,a4paper]{amsart}

\newcommand\TABLE{7}
\newcommand\RankOne{2}
\newcommand\Neighbor{Corollary~4.8}
\newcommand\obtuse{Theorem~4.5}

\usepackage{amsfonts, amssymb, amsthm}

\usepackage[alphabetic,nobysame]{amsrefs}

\catcode`\@=11

\def\@settitle{%
  \baselineskip14\p@\relax
  {\Large\bfseries
    \@title}}

\def\@setauthors{%
  \begingroup
  \def\thanks{\protect\thanks@warning}%
  \trivlist \footnotesize \@topsep45\p@\relax \advance\@topsep by
  -\baselineskip
\item\relax \author@andify\authors \def\\{\protect\linebreak}%
  {\sc\fontsize{12}{10}\selectfont\authors}%
  \ifx\@empty\contribs \else ,\penalty-3 \space \@setcontribs
  \@closetoccontribs \fi
  \endtrivlist
  \endgroup
}

\def\@secnumfont{\bfseries}%

\def\section{\@startsection{section}{1}%
  \z@{.7\linespacing\@plus\linespacing}{.5\linespacing}%
  {\normalfont\bf}}

\renewcommand{\BibLabel}{%

  \Hy@raisedlink{\hyper@anchorstart{cite.\CurrentBib}\hyper@anchorend}%
  [\thebib]\hfill%
}

\DefineSimpleKey{bib}{arxiv}

\BibSpec{article}{%
  +{} {\PrintAuthors} {author} +{,} { \textit} {title} +{.} { } {part}
  +{:} { \textit} {subtitle} +{,} { \PrintContributions}
  {contribution} +{.} { \PrintPartials} {partial} +{,} { } {journal}
  +{} { \textbf} {volume} +{} { \PrintDatePV} {date} +{,} {
    \issuetext} {number} +{,} { \eprintpages} {pages} +{,} { }
  {status} +{,} { \PrintDOI} {doi} +{,} { \arxiv} {arxiv} +{}
  { \parenthesize} {language} +{} { \PrintTranslation} {translation}
  +{;} { \PrintReprint} {reprint} +{.} { } {note} +{.} {} {transition}
  +{} {\SentenceSpace \PrintReviews} {review} }

\newcommand{\arxiv}[1]{\href{http://arxiv.org/abs/#1}{{\tt
      arxiv:\hspace{0pt}#1}}}

\catcode`\@=12

\usepackage[utf8]{inputenc}
\usepackage{parskip}
\usepackage[in]{fullpage}
\usepackage[colorlinks,breaklinks,allcolors=blue]{hyperref}
\usepackage{tikz}
\usetikzlibrary{calc}
\usepackage[cmtip,matrix,arrow]{xy}
\usepackage{array}
\usepackage{aliascnt} \usepackage{enumitem}
\setlist[enumerate,1]{label=\textit{\alph*)},ref=\textit{\alph*)}}
\setlist[enumerate,2]{label=\textit{\roman*)},ref=\textit{\roman*)}}

\numberwithin{equation}{section}

\theoremstyle{plain}

\swapnumbers

\newtheorem{theorem}{Theorem}[section]

\newaliascnt{lemma}{theorem}
\newtheorem{lemma}[lemma]{Lemma}
\aliascntresetthe{lemma}

\newaliascnt{corollary}{theorem}
\newtheorem{corollary}[corollary]{Corollary}
\aliascntresetthe{corollary}

\newaliascnt{proposition}{theorem}
\newtheorem{proposition}[proposition]{Proposition}
\aliascntresetthe{proposition}

\theoremstyle{definition}
\newtheorem{definition}[theorem]{Definition}
\newtheorem*{example}{Example}
\newtheorem*{examples}{Examples}

\newtheorem*{remark}{Remark}
\newtheorem*{remarks}{Remarks}

\font\polish=plr10
\renewcommand\l{{\polish\char'252}}
\newcommand\pole{{\polish\char'246}}
\newcommand\polS{\'S}

\newcommand\cC{{\mathcal C}}
\newcommand\cF{{\mathcal F}}
\newcommand\cL{{\mathcal L}}
\newcommand\cO{{\mathcal O}}
\newcommand\cV{{\mathcal V}}
\newcommand\cX{{\mathcal X}}
\def\|#1|{\operatorname{#1}}

\newcommand\NN{{\mathbb N}}
\newcommand\QQ{{\mathbb Q}}

\newcommand\ZZ{{\mathbb Z}}
\renewcommand\a{\alpha}
\renewcommand\b{\beta}

\newcommand\<{\langle}
\renewcommand\>{\rangle}
\newcommand\auf{\twoheadrightarrow}
\newcommand\into{\hookrightarrow}
\newcommand\Aq{{\overline A}}

\newcommand\Gq{{\overline G}}
\newcommand\Hq{{\overline H}}

\newcommand\Vq{{\overline V}}
\newcommand\Xq{{\overline X}}
\newcommand\Yq{{\overline Y}}

\newcommand\pfeil{\rightarrow}

\newcommand\Gm{{\mathbf G_m}}
\newcommand\A{{\mathsf A}}
\newcommand\B{{\mathsf B}}

\newcommand\G{{\mathsf G}}
\renewcommand\P{{\mathbf P}}
\renewcommand\phi{\varphi}

\newcommand\PGL{{\operatorname{PGL}}}
\newcommand\SL{{\operatorname{SL}}}

\newcommand\SO{{\operatorname{SO}}}

\newcommand\Sp{{\operatorname{Sp}}}

\newcommand\leer{\varnothing}

\SelectTips{cm}{10}
\newcommand\cxymatrix[1]{\vcenter{\begin{xy}\xymatrix@=15pt{#1}\end{xy}}}

\def\tP/{{\ifmmode (p)\else$(p)$\fi}}
\def\tB/{{\ifmmode (b)\else$(b)$\fi}}
\def\tA/{{\ifmmode (a)\else$(a)$\fi}}
\def\tAA/{{\ifmmode (2a)\else$(2a)$\fi}}

\begin{document}

\title{Localization of spherical varieties}

\author{Friedrich Knop\newline FAU Erlangen-Nürnberg}
\address {Department Mathematik, Emmy-Noether-Zentrum,
FAU Erlangen-Nürnberg, Cauerstraße 11, 91058 Erlangen, Germany}
\email{friedrich.knop@fau.de}

\begin{abstract}

  We prove some fundamental structural results for spherical varieties
  in arbitrary characteristic. In particular, we study Luna's two
  types of localization and use it to analyze spherical roots, colors
  and their interrelation. At the end, we propose a preliminary
  definition of a $p$-spherical system.

\end{abstract}

\maketitle

\setlength{\mathindent}{60pt}

\section{Introduction}
\label{sec:Introduction}

Let $G$ be a connected reductive group defined over an algebraically
closed ground field $k$ of arbitrary characteristic $p$. A normal
$G$-variety $X$ is called \emph{spherical} if a Borel subgroup $B$ of
$G$ has an open orbit in $X$. In characteristic $0$, there exists by
now an extensive body of research on spherical varieties culminating
in a complete classification (Luna-Vust~\cite{LuVu},
Luna~\cite{LuTypA}, Losev~\cite{Losev}, Cupit-Foutou~\cite{CuFou}, and
Bravi-Pezzini~\cite{BrPezz1,BrPezz2,BrPezz3}).

In positive characteristic, much less work has been done. Most papers
dealing with spherical varieties in positive characteristic are
restricted to particular examples (like flag or symmetric varieties)
or other special classes of spherical varieties (like varieties
obtained by reduction mod $p$).

The present paper is part of a program to develop a general theory of
spherical varieties in arbitrary characteristic, leading possibly to a
classification, as well. In this sense, the old paper \cite{Hyd} on a
characteristic free approach to spherical embeddings is already part
of the program.

A crucial portion of Luna's theory of spherical varieties depends on
Akhiezer's classification, \cite{Akh}, of spherical varieties of rank
1. In the present paper we present the results which are independent
of that classification. On the other hand, in the companion paper
\cite{SpRoot} we determine all spherical varieties of rank $1$ in
arbitrary characteristic and present results whose proofs depend (so
far) on it.

More precisely, in the present paper we recover most results of Luna's
paper \cite{LuGC} on the ``big cell''. We start with generalizing
Luna's fundamental relations for the colors of a spherical variety. At
this point, we introduce additional data needed to describe a
spherical variety in positive characteristic. These are certain
$p$-powers $q_{D,\a}$ where $\a$ is a simple root and $D$ is a color
``moved by $\a$''. Our exposition of this part is different (and we
think simpler) than Luna's and seems to be new even in characteristic
zero.

Next, we define the notion of spherical roots as properly normalized
normal vectors to the valuation cone. Luna's method of viewing them as
weights attached to a wonderful variety does not generalize.

Next, we consider Luna's construction from \cite{LuGC} which is called
\emph{localization at $S$}. Basically, it consists in analyzing the
open Bia\l ynicki-Birula cell with respect to a dominant $1$-parameter
subgroup of $G$. Our results are more general than Luna's, even in
characteristic zero, since Luna restricts his attention to wonderful
varieties while we formulate everything for so-called toroidal
varieties. From this, we derive Luna's important result that the
colors are, to a large extent, already determined by the spherical
roots.

Then we consider a construction which is called \emph{localization at
  $\Sigma$}. This procedure amounts to analyzing $G$-orbits of a
toroidal variety. Since this technique is mostly classical only the
proof for the behavior of type-\tA/-colors is new. Unfortunately, our
results remain somewhat incomplete since it is unknown whether orbit
closures in toroidal spherical embeddings are normal or not.

We use localization at $\Sigma$ to prove the important non-positivity
result \autoref{cor:bound}. As opposed to Luna's proof who uses
Wasserman's tables of rank-2-varieties \cite{Wass} our proof is
conceptual.

In the final \autoref{sec:SphericalSystem} we are making an attempt to
generalize Luna's notion of a spherical system to positive
characteristic. This is a combinatorial structure describing the roots
and the colors of a spherical variety. As additional data we propose
the $p$-powers $q_{\a,D}$ mentioned above and we hope that, at least
for $p\ne2,3$, these data are enough to describe a spherical
variety. As for the axioms, we restrict ourselves to those which
immediately generalize axioms in characteristic zero. Conditions which
are only meaningful in positive characteristic (like bounding the
denominators of the pairings $\delta_D(\alpha)$) are deferred to
future work.

So additional axioms will have to be added on at a later stage.

\medskip

\noindent{\bf Acknowledgment:} I would like to thank Guido Pezzini for
many discussions on the matter of this paper. In particular, the main
idea in the proof of \autoref{prop:typeAcolorInequal} is due to him.

\medskip

\noindent{\bf Notation:} In the entire paper, the ground field $k$ is
algebraically closed. Its characteristic exponent is denoted by $p$,
i.e., $p=1$ if $\|char|k=0$ and $p=\|char|k$, otherwise. The group $G$
is connected reductive, $B\subseteq G$ a Borel subgroup, and
$T\subseteq B$ a maximal torus. Let $\Xi(T)=\Xi(B)$ be its character
group. The set of simple roots with respect to $B$ is denoted by
$S\subset\Xi(T)$.

A rational function $f$ on $X$ is $B$-semiinvariant if there is a
character $\chi_f\in\cX(B)=\cX(T)$ such that
$f(b^{-1}x)=\chi_f(b)f(x)$ for all $b\in B$ and generic $x\in X$. If
$X$ is spherical, the character $\chi_f$ determines $f$ up to a
non-zero scalar. Let $\Xi(X)\subseteq \Xi(T)$ be the set of characters
of the form $\chi_f$. It is a finitely generated abelian group whose
rank is called the rank of $X$. We also use
$\Xi_\QQ(X):=\Xi(X)\otimes\QQ$ and $\Xi_p(X):=\Xi(X)\otimes\ZZ_p$ with
$\ZZ_p:=\ZZ[\frac1p]$.

\section{Colors}\label{sec:Colors}

Many properties of a spherical variety are determined by two sets of
data and their interrelation: colors and valuations. We start with
colors. Our results generalize those of \cite{LuGC} in characteristic
zero but the approach is different. We do not use compactifications
but use the completeness of flag varieties instead.

Let $X$ be a spherical $G$-variety with group of characters $\Xi(X)$
and let
\begin{equation}
  N_\QQ(X)=\|Hom|(\Xi(X),\QQ).
\end{equation}
A \emph{color} of $X$ is an irreducible divisor which is $B$- but not
$G$-invariant. Every color $D$ produces an element $\delta_D\in
N_\QQ(X)$ by
\begin{equation}
  \delta_D(\chi_f):=v_D(f)
\end{equation}
for all $B$-semiinvariants $f$. Here $v_D$ is the valuation of $k(X)$
attached to $D$. The color $D$ is, in general, not uniquely determined
by $\delta_D$.

Since $X$ is spherical, we can choose a point $x_0\in X$ such that
$Bx_0$ is open and dense in $X$. Let $\Delta(X)$ be the set of colors
of $X$. Since every color intersects the open $G$-orbit $Gx_0$ we have
$\Delta(X)=\Delta(Gx_0)=\Delta(G/H)$ where $H=G_{x_0}$ is the isotropy
subgroup scheme of $x_0$. We start by recalling a well-known formula
for the number of colors.

\begin{proposition}\label{prop:NumberColors}

  Let $G$ be a semisimple group and $H\subseteq G$ a spherical
  subgroup. Then
  \begin{equation}\label{eq:40}
    \#\Delta(G/H)=\|rk|G/H+\|rk|\Xi(H).
  \end{equation}

\end{proposition}

\begin{proof}

  We compute the Picard group in two different ways. Set
  $X=G/H$. First, we have an exact sequence
  \begin{equation}
    \Xi(G)\pfeil\|Pic|^GX\pfeil\|Pic|X.
  \end{equation}
  The left hand group is trivial since $G$ is semisimple. The cokernel
  of the right hand homomorphism is torsion by Sumihiro \cite{Sum}. On
  the other hand $\|Pic|^GX=\Xi(H)$. Thus
  $\|rk|\|Pic|X=\|rk|\Xi(H)$. Now let $X_0=Bx_0\subseteq X$ be the
  open $B$-orbit. Then the colors are the irreducible components of
  $X\setminus X_0$. Thus we have an exact sequence
  \begin{equation}
    k^\times=\cO(X)^\times\pfeil\cO(X_0)^\times
    \pfeil\ZZ^a\pfeil\|Pic|X\pfeil\|Pic|X_0=0
  \end{equation}
  where $a=\#\Delta(X)$. By definition
  $\|rk|\cO(X_0)^\times/k^\times=\|rk|X$. Thus
  $\|rk|\|Pic|X=a-\|rk|X$.
\end{proof}

Given a simple root $\a\in S$, one can construct colors as follows:
let $P_\a\subseteq G$ be the minimal parabolic subgroup corresponding
to $\a$. Then $P_\a x_0$ is an open $B$-stable subset of $X$ which,
according to \cite{BorSub}*{3.2}, decomposes into at most three
$B$-orbits. One of them is the open $B$-orbit $Bx_0$; the others are
of codimension $1$ in $X$, hence their closures are colors. We say
that these colors are \emph{moved by $\a$}. Clearly, this just means
that $P_\a D\ne D$. In particular, every color is moved by some (not
necessarily unique) simple root.

A more precise description is as follows. Let
$H_\a:=(P_\a)_{x_0}=H\cap P_\a$ such that $P_\a x_0=P_\a/H_\a$. Then
the $B$-orbits in $P_\a x_0$ correspond to $H_\a^{\|red|}$-orbits in
$B\backslash P_\a\cong\P^1$. Let $\Hq_\a$ denote the image of
$H_\a^{\|red|}$ in $\|Aut|\P^1\cong\PGL(2)$. Then, up to conjugation,
there are four possibilities for $\Hq_\a$:
\begin{equation}\label{eq:2}
  \begin{array}{lll}

    \text{Type of $\a$}&\Hq_\a&\text{colors}\\
    \noalign{\smallskip\hrule\smallskip}
    \tP/&G_0&\text{---}\\
    \tB/&S_0U_0&D\\
    \tA/&T_0&D,D'\\
    \tAA/&N_0&D

  \end{array}
\end{equation}

Here $G_0=\PGL(2)$. The subgroups $B_0$, $U_0$, and $T_0$ of $G_0$ are
Borel subgroup, a maximal unipotent subgroup, and a maximal torus,
respectively. Moreover, $S_0\subseteq T_0$ and
$N_0=N_{G_0}(T_0)$. Thus, the set of simple roots decomposes as a
disjoint union according to their type:
\begin{equation}
  S=S^\tP/\cup S^\tB/\cup S^\tA/\cup S^\tAA/.
\end{equation}
Observe that $\a\in S^\tP/$ if and only if the open $B$-orbit $Bx_0$
is $P_\a$-invariant. Thus, $S^\tP/$ is the set of simple roots of the
parabolic $P_X$, the stabilizer of the open $B$-orbit.

Let $D$ be a color moved by $\a$. Then the morphism
\begin{equation}
  \phi_{D,\a}:P_\a\times^BD\pfeil X
\end{equation}
is generically finite. Its separable degree is $1$, i.e.,
$\phi_{D,\a}$ is bijective if and only if $\a$ is of type \tB/ or
\tA/. It is $2$ for $\a$ of type \tAA/. The inseparable degree of
$\phi_{D,\a}$ will be denoted by $q_{D,\a}\in p^\NN$.

\begin{example}

  Assume $p>3$ and let $P\subseteq G$ be a subgroup scheme which
  contains ${}^-B$, the Borel subgroup which is opposite to $B$. In
  \cite{Wenz0,Wenz}, Wenzel shows that such subgroup schemes are
  classified by functions $f:S\to\ZZ_{\ge0}\cup\{\infty\}$ where
  $f(\a)$ can be defined as the supremum of the set of all
  $n\in\ZZ_{\ge0}$ such that $P$ contains the $n$-th Frobenius kernel
  of $P_\a$. See also \cite{HaLau} for a simplified account. Now let
  $X=G/P$, a complete homogeneous $G$-variety. Then the following is
  easy to see: A simple root $\a\in S$ is of type \tP/ if and only if
  $f(\a)=\infty$. All other simple roots are of type \tB/ and they all
  move a different divisor $D_\a$. Moreover,
  $q_{D_\a,\a}=p^{f(\a)}$. In particular, this shows that in this
  example the numbers $q_{D,\a}$ may be arbitrary $p$-powers.

\end{example}

To formulate the following permanence property we renormalize
$\delta_D$ as follows:
\begin{equation}
  \delta_D^{(\a)}:=q_{D,\a}\delta_D.
\end{equation}

\begin{lemma}\label{lem:finitemor}

  Let $\pi:X_1\pfeil X_0$ be a finite surjective equivariant morphism
  between spherical $G$-varieties, let $E$ be a color of $X_1$ and let
  $D=\pi(E)$ be its image in $X_0$. Let, moreover, $\a\in S$ be a
  simple root moving $E$ (and $D$). Then
  $\delta_D^{(\a)}=\delta_E^{(\a)}$.

\end{lemma}

\begin{proof}

  We consider first the case that $\a$ is of type \tA/ or \tB/ for
  $X_0$. Then its type for $X_1$ is the same. Moreover, both
  $\phi_{D,\a}$ and $\phi_{E,\a}$ are bijective and, as an equality of
  divisors, $\pi^{-1}(D)=qE$ where $q$ is some $p$-power. Thus,
  $\delta_E=q\delta_D$. Now consider the diagram
  \begin{equation}
    \cxymatrix{
      P_\a\times^B\pi^{-1}(D)\ar[r]\ar[d]&X_1\ar[d]\\
      P_\a\times^BD\ar[r]^(.6){\phi_{D,\a}}&X_0\\}
  \end{equation}
  It is cartesian, hence both horizontal arrows have the same
  (inseparable) degree, namely $q_{D,\a}$. On the other hand, the top
  arrow has degree $q\,q_{E,\a}$. Hence
  \begin{equation}
    \delta^{(\a)}_D=q_{D,\a}\delta_D=qq_{E,\a}\delta_D=q_{E,\a}\delta_E=
    \delta_E^{(\a)}.
  \end{equation}

  Now assume that $\a$ is of type \tAA/ for $X_0$. Then there are two
  cases. If $\pi^{-1}(D)^{\|red|}$ is irreducible then $\a$ is of type
  \tAA/ for $X_0$, as well. Moreover, the degree of both horizontal
  arrows is now $2q_{D,\a}$. From here one argues as above.

  The second case is when $\a$ of type $\tA/$ for $X_0$. Then
  $\pi^{-1}(D)^{\|red|}=E_1\cup E_2$ has two components. As divisors,
  we have $\pi^{-1}(D)=q_1E_1+q_2E_2$. Thus $\delta_{E_1}=q_1\delta_D$
  and $\delta_{E_2}=q_2\delta_D$. Moreover, as above, we get
  \begin{equation}
    2\delta_D^{(\a)}=2q_{D,\a}\delta_D=
    q_1q_{E_1,\a}\delta_D+q_2q_{E_2,\a}\delta_D=
    \delta_{E_1}^{(\a)}+\delta_{E_2}^{(\a)}
  \end{equation}
  Now we claim that actually $\delta_{E_1}^{(\a)}=\delta_{E_2}^{(\a)}$
  which would prove our assertion.

  To prove the claim, we may assume that $X_0=G/H_0$ and $X_1=G/H_1$
  are homogeneous. Moreover, the cases proved above allow to replace
  $H_0$ and $H_1$ by $H_0^{\|red|}$ and $H_1^{\|red|}$,
  respectively. We can even replace $H_1$ be its connected component
  of unity since $E$ cannot split any further (otherwise $D$ would
  split into more than two components). Then $H_1$ is normal in $H_0$
  and $\pi$ is the quotient by the finite group
  $\Gamma:=H_0/H_1$. Since $\Gamma$ acts transitively on the connected
  components of the fibers of $\pi$, there is an element $g\in\Gamma$
  which maps $E_1$ to $E_2$ which proves the claim.
\end{proof}

For any simple root $\a$ let $\a^r\in N_\QQ(X)$ be the restriction of
$\a^\vee$ to $\Xi_\QQ(X)$, i.e.,
\begin{equation}\label{eq:3}
  \a^r(\chi)=\<\chi,\a^\vee\>\quad\hbox{for all $\chi\in\Xi_\QQ(X)$}
\end{equation}

\begin{proposition}\label{prop:ColorRelations}

  Fix a simple root $\a\in S$. Then the following relations hold:
  \begin{equation}\label{eq:5}
    \begin{array}{ll}

      \text{Type of $\a$}\\
      \noalign{\smallskip\hrule\smallskip}
      \tP/&\a^r=0\\
      \tB/&\delta_D^{(\a)}=\a^r\\
      \tA/&\delta_D^{(\a)}+\delta_{D'}^{(\a)}=\a^r\\
      \omit\hfill&\a\in\Xi_p(X), \delta_D^{(\a)}(\a)=\delta_{D'}^{(\a)}(\a)=1\\
      \tAA/&\delta_D^{(\a)}=\frac12\a^r

    \end{array}
  \end{equation}

\end{proposition}

\begin{proof}

  Let $X=G/H$ be homogeneous and put $X_1=G/H^{0,\|red|}$. We claim
  that it suffices to prove the assertions for $X_1$. Indeed, if $\a$
  is of type \tP/ for $X$ then the same holds for $X_1$ and the claim
  follows from $\Xi_\QQ(X)=\Xi_\QQ(X_1)$. If $\a$ is not of type \tP/
  for $X$ and $X_1$ then the claim follows immediately from
  \autoref{lem:finitemor} if $\a$ has the same type for $X$ and
  $X_1$. Otherwise, $\a$ is of type \tAA/ for $X$ (moving one color
  $D$) and of type \tA/ for $X_1$ (moving two colors $E_1$,
  $E_2$). But then \autoref{lem:finitemor} implies
  \begin{equation}
    \delta_D^{(\a)}={\textstyle\frac12}(\delta_{E_1}^{(\a)}+\delta_{E_2}^{(\a)})
    ={\textstyle\frac12}\a^r
  \end{equation}
  proving the claim.

  Thus we may assume that $H$ is connected and reduced. Then consider
  the following diagram
  \begin{equation}\label{eq:6}
    \qquad\cxymatrix{&G\ar@{>>}[dl]_{p_1}\ar@{>>}[dr]^{p_2}&\\
      \llap{$X=$ }G/H&&B\backslash G\rlap{ $=:\cF$}\\}
  \end{equation}
  Both morphisms $p_1$ and $p_2$ are smooth with connected
  fibers. Therefore, an irreducible $B$-stable divisors $D\subset X$
  corresponds to an irreducible $H$-stable divisor $E\subset
  \cF$. Moreover, any $B$-semiinvariant rational function $f$ on $X$
  corresponds to an $H$-invariant rational section $s$ of the
  homogeneous line bundle $\cL_\chi$ (with $\chi=\chi_f$) on
  $\cF$. Furthermore, $(D,f)$ is related to $(E,s)$ by
  \begin{equation}\label{eq:7}
    v_E(s)=v_D(f)=\delta_D(\chi).
  \end{equation}

  Now consider the $\P^1$-fibration $\pi:\cF=B\backslash G\pfeil
  \cF_\a:=P_\a\backslash G$. Moreover, let $y\in\cF$ be in the open
  $H$-orbit and let $F\cong\P^1$ be the fiber through $y$.

  Assume first that $\a$ is of type \tP/. Then the open $B$-orbit in
  $X$ is $P_\a$-stable which translates into the open $H$-orbit in
  $\cF$ being the preimage of an open set in $P_\a\backslash G$. But
  then $\cL_\chi$ is a pull-back from $P_\a\backslash G$ which implies
  $\<\chi_f,\a^\vee\>=0$.

  Now assume that $\a\in S^\tB/$. Then $E$ is the only $H$-invariant
  divisor mapping dominantly onto $\cF_\a$. Moreover, since $E\cap F$
  consists of a single point, the map $E\pfeil \cF_\a$ is generically
  bijective, hence purely inseparable. Its degree is $q_{D,\a}$. Thus
  we get
  \begin{equation}\label{eq:8}
    \<\chi,\a^\vee\>=\|deg|\cL_\chi|_F=(s)\cdot F=v_E(s)E\cdot
    F=q_{D,\a}\delta_D(\chi)=\delta_D^{(\a)}(\chi)
  \end{equation}
  proving the assertion.

  If $\a\in S^\tA/$ then there are two divisors $E$, $E'$ mapping
  generically injectively to $\cF_\a$ with degree $q_{D,\a}$ and
  $q_{D',\a}$, respectively. Then
  \begin{equation}\label{eq:9}
    \<\chi,\a^\vee\>=(s)\cdot F=(v_E(s)E+v_{E'}(s)E')\cdot
    F=q_{D,\a}\delta_D(\chi)+q_{D',\a}\delta_{D'}(\chi).
  \end{equation}

  Now we prove $\a\in\Xi_p(X)$. By construction, there is an
  equivariant morphism $P_\a x_0\auf \PGL(2)/\tilde H$ with $\tilde
  H^{\|red|}=T_0$. Thus the pull-back of any non-constant
  $B_0$-semi\-invariant is a $B$-semiinvariant with character $q_0\a$
  for some $p$-power $q_0$.  The $H_\a$-linearization of
  $\cL_{q_0\a}|_F$ factors through a $\PGL(2)$-linearization. One
  reason is, e.g., that $\cL_{-\a}$ is the relative canonical bundle
  of the fibration $\pi$. This implies that $s|_F$ has two zeroes of
  the same multiplicity on $F$. Hence
  $\delta_D^{(\a)}(\a)=\delta_{D'}^{(\a)}(\a)$ and therefore both are
  equal to $1$.

  Finally, assume that $\a\in S^\tAA/$. Then there is one divisor $E$
  mapping generically $2:1$ to $\cF_\a$.The degree of inseparability
  of this map is $q_{D,\a}$. Then $E\cdot F=2q_{D,\a}$ and therefore
  \begin{equation}\label{eq:10}
    \<\chi,\a^\vee\>=(s)\cdot F=v_E(s)E\cdot
    F=2q_{D,\a}\delta_D(\chi)=2\delta_D^{(\a)}(\chi),
  \end{equation}
  as claimed.
\end{proof}

We note the following consequence:

\begin{corollary}\label{cor:Parity}

  Let $p\ne2$ and let $\a\in S$ be of type \tAA/. Then
  $\a\not\in\Xi_p(X)$ and $\<\chi,\a^\vee\>$ is even for all
  $\chi\in\Xi(X)$.

\end{corollary}

\begin{proof}

  We keep the notation of the proof of
  \autoref{prop:ColorRelations}. Let $N_0=\<s_0\>T_0$ and let
  $n\in\ZZ$ with $n\a\in\Xi(X)$. Then $s_0$ acts on the
  $T_0$-invariant section of $\cL_{q_0\a}|_F$ by multiplication with
  $(-1)^n$. Hence $n$ is even and $\a\not\in\Xi_p(X)$. The rest
  follows directly from \autoref{prop:ColorRelations}.
\end{proof}

Now we analyze the case where a color is moved by more than one simple
root.

\begin{lemma}\label{lem:Multiple}

  Let $D$ be a color which is moved by two distinct simple roots
  $\a_1$ and $\a_2$. Then either $\a_1,\a_2\in S^\tB/$ or
  $\a_1,\a_2\in S^\tA/$. In the latter case let $D'$ and $D''$ be the
  second color moved by $\a_1$ and $\a_2$, respectively. Then $D'\ne
  D''$.

\end{lemma}

\begin{proof}

  Clearly, neither $\a_1$ nor $\a_2$ is of type \tP/. Recall from
  \cite{BorSub}*{\S2}, that any $B$-orbit on $X$ has a rank attached
  to it. Moreover, if $\a\in S$ moves the color $D$ then
  $\|rank|D=\|rank|X$ if $\a$ is of type \tB/ and
  $\|rank|D=\|rank|X-1$ in case $\a$ is of type \tA/ or \tAA/
  (\cite{BorSub}*{\S2 and Lemma 3.2}). This entails that $\a_1$,
  $\a_2$ are either both of type \tB/ or both of type \tA/ or \tAA/.

  Suppose they are both of type \tAA/. Then, since
  $\a_1\in\Xi_\QQ(X)$,
  \begin{equation}\label{eq:11}
    0<q_{D,\a_1}^{-1}=q_{D,\a_1}^{-1}\,\textstyle{\frac12}\a_1^r(\a_1)=
    \delta_D(\a_1)=
    q_{D,\a_2}^{-1}\,\textstyle{\frac12}\a_2^r(\a_1)\le0.
  \end{equation}
  Similarly, suppose $\a_1$ is of type \tA/ and $\a_2$ is of type
  \tAA/. Then
  \begin{equation}\label{eq:12}
    0<q_{D,\a_1}^{-1}=\delta_D(\a_1)=
    q_{D,\a_2}^{-1}\textstyle{\frac12}\a_2^r(\a_1)\le0.
  \end{equation}
  This finishes the proof of the first part.

  Now let both $\a_1$ and $\a_2$ be of type \tA/ and suppose
  $D'=D''$. Then
  \begin{equation}\label{eq:13}
    0<\frac{q_{D,\a_2}}{q_{D,\a_1}}+\frac{q_{D',\a_2}}{q_{D',\a_1}}=
    q_{D,\a_2}\delta_D(\a_1)+q_{D',\a_2}\delta_{D'}(\a_1)=
    \a_2^r(\a_1)\le0.
  \end{equation}
\end{proof}

\begin{examples} 1. Let $G=\SL(2)\times\SL(2)$ and $H=\SL(2)$ embedded
  into $G$ via $\|id|\times F_q$ where $F_q$ is a Frobenius
  morphism. Then $X:=G/H$ has only one color $D$. Moreover both simple
  roots $\a_1$, $\a_2$ are of type \tB/ and $D$ is moved by both of
  them. Furthermore, $q_{D,\a_1}=1$ and $q_{D,\a_2}=q$ which shows
  that $q_{D,\a}$ may depend on $\a$.

  2. Let $G=\SL(3)$, let $q$ be a $p$-power, and let $H$ be the
  subgroup consisting of the matrices
  \begin{equation}
    \begin{pmatrix} t^{q+2}&&\\&t^{q-1}&\\&&t^{-2q-1}\end{pmatrix}
    \cdot
    \begin{pmatrix}1&x&y\\&1&x^q\\&&1\end{pmatrix}
    \quad\text{with $t\in\Gm$ and $x,y\in{\bf G}_a$.}
  \end{equation}
  Then both simple roots are of type \tA/ and there are three colors
  $D_0,D_1,D_2$ where $\a_i$ moves $D_0$ and $D_i$. Furthermore,
  $q_{D_1,\a_1}=q_{D_0,\a_2}=q_{D_2,\a_2}=1$ while
  $q_{D_0,\a_1}=q$. The values $\delta_D(\a_i)$ are given by the
  following table:
  \begin{equation}
    \setlength{\extrarowheight}{5pt}
    \begin{array}{l|ccc}

      &\delta_{D_0}&\delta_{D_1}&\delta_{D_2}\\
      \hline
      \a_1&q^{-1}&1&-1-q^{-1}\\
      \a_2&1&-q-1&1

    \end{array}
  \end{equation}
  So indeed $q\delta_{D_0}+\delta_{D_1}=\a_1^r$ and
  $\delta_{D_0}+\delta_{D_2}=\a_2^r$.

  Both examples show that $G$ contains for $p\ge2$ infinitely many
  conjugacy classes of self-normalizing spherical subgroups, a
  phenomenon which does not occur in characteristic zero.

\end{examples}

\begin{remark}
  The Lemma shows in particular, that one can assign unambiguously a
  type to any color. Thereby, one gets a decomposition
  \begin{equation}
    \Delta(X)=\Delta^\tB/(X)\cup\Delta^\tA/(X)\cup\Delta^\tAA/(X).
  \end{equation}
\end{remark}

\section{Spherical roots}
\label{sec:Roots}

For a spherical variety $X$, let $\cV(X)$ be the set of $G$-invariant
$\QQ$-valued valuations of the field $k(X)$. The map
\begin{equation}\label{eq:14}
  \cV(X)\pfeil N_\QQ(X)=\|Hom|(\Xi(X),\QQ):
  v\mapsto\big(\chi_f\mapsto v(f)\big)
\end{equation}
is injective (\cite{Hyd}*{1.8}). According to \cite{Hyd}*{5.3}, it
identifies $\cV(X)$ with a a finitely generated convex rational cone
inside $N_\QQ(X)$ which contains the image of the antidominant Weyl
chamber under the projection $\|Hom|(\Xi(T),\QQ)\auf N_\QQ(X)$. This
can be phrased in terms of the dual cone $\cV(X)^\vee$ of $\cV(X)$: it
is a finitely generated rational convex cone in $\Xi_\QQ(X)$ with
$\cV(X)=(\cV(X)^\vee)^\vee$ and $-\cV(X)^\vee\subseteq\QQ_{\ge0}S$
where $\QQ_{\ge0}S$ is the cone generated by the simple roots of
$G$. In particular, $-\cV(X)^\vee$ is a pointed cone. Thus, it has a
canonical set of generators:

\begin{definition}

  An element $\sigma\in\Xi_\QQ(X)$ is called a \emph{spherical root of
    $X$} if

  \begin{itemize}

  \item $\QQ_{\ge0}\sigma$ is an extremal ray of $-\cV(X)^\vee$ (thus
    $\sigma\in\QQ S$) and

  \item $\sigma$ it is a primitive element of $\ZZ S\cap\Xi_p(X)$.

  \end{itemize}

  The set of spherical roots is denoted by $\Sigma(X)$.

\end{definition}

Clearly, each extremal ray of $-\cV(X)^\vee$ contains a unique
spherical root. Moreover, the spherical roots determine the valuation
cone via
\begin{equation}\label{eq:15}
  \cV(X)=\{v\in N_\QQ(X)\mid \sigma(v)\le0
  \hbox{ for all $\sigma\in\Sigma(X)$}\},
\end{equation}
and are in bijection with faces of codimension $1$ of $\cV(X)$.

The normalization for a spherical root is chosen such that the
following statement holds:

\begin{lemma}

  Let $\phi: X_1\pfeil X_2$ be a morphism of spherical varieties which
  is either purely inseparable or a quotient by a central subgroup
  scheme of $G$. Then $\Sigma(X_1)=\Sigma(X_2)$.

\end{lemma}

\begin{proof}

  If $\phi$ is purely inseparable then clearly $\cV(X_1)=\cV(X_2)$ and
  $\Xi_p(X_1)=\Xi_p(X_2)$. Hence $\Sigma(X_1)=\Sigma(X_2)$.

  Now let $\phi$ be the quotient by $A\subseteq Z(G)$ (which might be
  positive dimensional). Then
  \begin{equation}
    \Xi_p(X_2)=\{\chi\in\Xi_p(X_1)\mid\chi|_A=1\}.
  \end{equation}
  Since roots of $G$ are trivial on $A$, this implies
  \begin{equation}
    \ZZ S\cap\Xi_p(X_2)=\ZZ S\cap\Xi_p(X_1).
  \end{equation}
  Moreover, $N_\QQ(X_2)$ is a quotient of $N_\QQ(X_1)$ and $\cV(X_1)$
  is the preimage of $\cV(X_2)$ (follows, e.g., from
  \cite{Hyd}*{Thm.~6.1}). This implies $\Sigma(X_2)=\Sigma(X_1)$.
\end{proof}

\section{Localization at $S$}
\label{sec:BigCell}

There are two types of constructions, called localization at $S$ and
at $\Sigma$, respectively, which both allow to reduce a spherical
variety to a simpler one. In this section we describe localization at
$S$ which, in characteristic zero, was first introduced by Luna in
\cite{LuGC}. To this end, we first recall and prove some properties of
the Bia\l ynicki-Birula decomposition, \cite{BiaBir}, of a
$\Gm$-variety.

Let $X$ be a complete normal $\Gm$-variety. Then for any $x\in X$, the
limit
\begin{equation}
  \pi(x):=\lim_{t\pfeil0}t\cdot x\in X
\end{equation}
exists and is a $\Gm$-fixed point. Thus, letting $F$ to be the set of
connected components of the fixed point set $X^\Gm$ we get a
set-partition of $X$ by putting
\begin{equation}\label{eq:16}
  X_Z:=\{x\in X\mid \pi(x)\in Z\}.
\end{equation}
These are the Bia\l ynicki-Birula cells which are indexed by
$F$. Except when $X$ is smooth or projective they are, in general, not
very well behaved. One cell is always good, though:

\begin{proposition}\label{prop:Source}

  Let $X$ be a complete normal $\Gm$-variety. Then there is a unique
  connected component $S$ of $X^\Gm$ (the {\rm source of $X$}) such
  that $X_S$ is open. Moreover, the map $\pi_S:=\pi_S:X_S\pfeil S$ is
  affine and a categorical quotient by $\Gm$. In particular, the
  source $S$ is irreducible and normal.

\end{proposition}

\begin{proof}

  The statement is well-known. For example, it follows from the theory
  in \cite{BiaSwi}: Let $S$ be the source, i.e., the connected
  component of $X^\Gm$ such that $\pi(x)\in S$ for $x\in X$
  generic. By \cite{BiaSwi}*{Prop.~2.3}, the set $A^+=\{S\}$ defines a
  sectional set. Now the assertion is \cite{BiaSwi}*{Thm.~1.5}.
\end{proof}

\begin{lemma}\label{lem:GenFiber}

  Let $X$ be as above. Then the general fibers of $\pi_S$ are
  irreducible and generically reduced.

\end{lemma}

\begin{proof}

  Since $X_S$ is normal, the generic fiber of $\pi_S$ is geometrically
  unibranch (\cite{EGA4}*{6.15.6}). Since all irreducible components
  contain the $\Gm$-fixed point, it follows that the generic fiber is
  geometrically irreducible. Thus, there is an open subset of $S$ over
  which all fibers are geometrically irreducible, as well
  (\cite{Jou}*{Thm.~4.10}).

  The second property follows from the fact that $\pi_S$ is a
  categorical quotient. This entails that $k(S)$ is separable inside
  $k(X_S)$. Therefore, $\pi_S$ is smooth generically on $X_S$.
\end{proof}

There is a second well-behaved cell. For this, let $X^-:=X$ but with
the opposite $\Gm$-action: $t\ast x:=t^{-1}\cdot x$. Then the source
$T$ of $X^-$ is called the \emph{sink of $X$}. It is characterized by
the fact that $\pi(x)\in T$ implies $x\in T$. Thus, $T=X_T$ is a Bia\l
ynicki-Birula cell by itself. Symmetrically, $S$ is the sink of $X^-$
and therefore $X^-_S=S$.

Now assume that $X$ is a $G$-variety for some connected reductive
group $G$ and that the $\Gm$-action is induced by a $1$-parameter
subgroup $\lambda:\Gm\pfeil G$. Then we put $X^\lambda:=S$ and
$X_\lambda:=X_S$ and $\pi_\lambda=\pi_S$. Observe that $X^\lambda$ is
not the entire fixed point set of $\lambda(\Gm)$ but only a very
special component.

Let $G^\lambda:=C_G(\lambda(\Gm))$ be the fixed point set under the
conjugation action and let
\begin{equation}\label{eq:18}
  G_\lambda:=\{g\in G\mid
  \pi_G(g):=\lim_{t\pfeil0}\lambda(t)g\lambda(t)^{-1}\hbox{ exists}\}.
\end{equation}
Then $G_\lambda$ is a parabolic subgroup with Levi complement
$G^\lambda$ and the map $\pi_G:G_\lambda\pfeil G^\lambda$ is the
natural homomorphism with kernel $G^u_\lambda$, the unipotent radical
of $G_\lambda$. The following lemma is well-known, e.g., the proof
given in \cite{LuGC} carries over verbatim to positive characteristic.

\begin{lemma}

  The open cell $X_\lambda$ is $G_\lambda$-invariant and
  $\pi_\lambda:X_\lambda\pfeil X^\lambda$ is $G_\lambda$-equivariant
  where $G_\lambda$ acts on $X^\lambda$ via $\pi_G:G_\lambda\pfeil
  G^\lambda$, i.e., $\pi_\lambda(gx)=\pi_G(g)\pi_\lambda(x)$ for all
  $g\in G_\lambda$ and $x\in X_\lambda$. Moreover, $X^\lambda$
  consists of fixed points for the opposite unipotent radical
  ${}^-\!G^u_\lambda$.

\end{lemma}

The next lemma provides a link between closed orbits in $X$ and closed
orbits in $X^\lambda$:

\begin{lemma}\label{lem:ClosedOrbit}

  Let $X$ be a complete normal $G$-variety and let $Z\subseteq
  X^\lambda$ be a closed $G^\lambda$-orbit. Then $GZ\subseteq X$ is a
  closed $G$-orbit with $GZ\cap X^\lambda=Z$.

\end{lemma}

\begin{proof}

  Since $Z$ is complete and homogeneous, it contains a unique fixed
  point $z$ for ${}^-\!B^\lambda$, the Borel subgroup opposite to
  $B^\lambda$. Since $z$ is in the source, it is a
  ${}^-\!G^u_\lambda$-fixed point. So $z$ is fixed by
  ${}^-\!B={}^-\!B^\lambda\,{}^-\!G^u_\lambda$ which implies that
  $Gz=GZ$ is a complete, hence closed $G$-orbit of $X$. Moreover,
  since $\lambda$ is dominant, we have $z\in (GZ)^\lambda$ and
  therefore $GZ\cap X^\lambda=(Gz)^\lambda=G^\lambda z=Z$.
\end{proof}

Recall that a complete spherical $G$-variety $X$ is called
\emph{toroidal} if no color of $X$ contains a $G$-orbit.

\begin{proposition}\label{prop:LocalizationAtS1}

  Let $\lambda$ be a dominant $1$-parameter subgroup and let $X$ be a
  complete toroidal spherical variety. Then $X^\lambda$ is a complete
  toroidal spherical $G^\lambda$-variety.

\end{proposition}

\begin{proof}

  This is basically \cite{LuGC}*{Prop.~1.4}. We recall the proof and
  check that it is characteristic free.

  First of all, completeness of $X^\lambda$ is clear while
  irreducibility and normality were obtained in
  \autoref{prop:Source}. Moreover, the dominance of $\lambda$ implies
  $B\subseteq G_\lambda$. Hence the open $B$-orbit of $X$ is contained
  in $X_\lambda$. Its image in $X^\lambda$ is an open
  $B^\lambda$-orbit. Thus, $X^\lambda$ is spherical, as well.

  Now let $D\subseteq X^\lambda$ be a color containing the closed
  $G^\lambda$-orbit $Z$. Then $\pi_\lambda^{-1}(D)$ contains a unique
  component $E'$ which maps onto $D$ (by \autoref{lem:GenFiber} and
  the fact that $D$ meets the open $G^\lambda$-orbit). Then $E$, the
  closure of $D$ in $X$, is a color which contains $Z$. Since $GZ$ is
  a closed $G$-orbit by \autoref{lem:ClosedOrbit} we have
  $GZ=\overline{BZ}$ and therefore $GZ\subseteq E$. It follows that
  $E$ is $G$-invariant since $X$ is toroidal. From this we get that
  $D=\pi_\lambda(E)$ is $G^\lambda$-invariant in contradiction to $D$
  being a color.
\end{proof}

A toroidal spherical variety $X$ determines a (pointed) fan
$\cF=\cF(X)$ in $N_\QQ(X)$ whose support is $\cV(X)$. More precisely,
for any $G$-orbit $Y\subseteq X$ the invariant valuations whose center
in $X$ contain $Y$ form a cone $\cC_Y\subseteq\cV(X)$. Then $\cF$ is
the collection of cones of the form $\cC_Y$.

The fan $\cF$ is precisely the piece of information needed to
reconstruct $X$ from its open $G$-orbit $X_0$. In fact, $X$ is the
compactification of $X_0$ corresponding to the colored fan
$(\cF,\leer)$. See \cite{Hyd} for more details.

Let $\lambda$ be a dominant $1$-parameter subgroup and let $X$ be a
complete toroidal spherical variety with associated fan $\cF$. Then
$\lambda$ induces, via restriction to $\Xi_\QQ(X)$, an element
$\lambda^r\in N_\QQ(X)$. The dominance of $\lambda$ implies
$-\lambda^r\in\cV(X)$. Thus we can consider the fan
\begin{equation}
  \tilde\cF^\lambda:=\{\cC+\QQ_{\ge0}\lambda\mid-\lambda^r\in\cC\in\cF\}
\end{equation}
One may think of $\tilde\cF^\lambda$ as the restriction of $\cF$ to a
neighborhood of $-\lambda^r$. Visibly, this fan is not pointed since
all its members contain the line $\QQ\lambda^r$. More precisely, let
$\cC(\lambda)$ be the unique cone in $\cF$ such that $-\lambda^r$ is
contained in its relative interior. Then
$V(\lambda):=\<\cC(\lambda)\>_\QQ=\cC(\lambda)+\QQ_{\ge0}\lambda^r$ is
the unique element of $\tilde\cF^\lambda$ which is a subspace. Thus,
\begin{equation}
  \cF^\lambda:=\{\cC/V(\lambda)\mid\cC\in\tilde\cF^\lambda\}.
\end{equation}
is a pointed fan which lives in the vector space $N_\QQ(X)/V(\lambda)$
and is called the \emph{localization of $\cF$ at $\lambda$}.

\begin{theorem}

  Let $\lambda$ and $X$ as in \autoref{prop:LocalizationAtS1}. Then
  $\Xi(X^\lambda)=\Xi(X)\cap V(\lambda)^\perp$,
  $N_\QQ(X^\lambda)=N_\QQ(X)/V(\lambda)$, and
  $\cF(X^\lambda)=\cF(X)^\lambda$.

\end{theorem}

\begin{proof}

  Let $z$ be an arbitrary ${}^-\!B$-fixed point in $X$. Then $z$
  corresponds to the complete orbit $Gz$ and therefore to a cone
  $\cC_{Gz}$ of maximal dimension in $\cF(X)$. Let $P$ be the
  parabolic which is opposite to the reduced stabilizer
  $G_z^{\|red|}$. Then the local structure theorem \cite{InvBew}*{1.2}
  asserts the existence of a normal affine $T$-variety $\Aq$ and a
  $T$-equivariant morphism $\phi_0:\Aq\pfeil X$ such that the morphism
  \begin{equation}
    \phi: P_u\times\Aq\pfeil X:(u,a)\mapsto u\phi_0(a)
  \end{equation}
  is finite onto an open neighborhood of $z$ in $X$. Moreover, the
  torus $T$ has an open orbit $A$ in $\Aq$ such that
  $\Xi_\QQ(A)=\Xi_\QQ(X)$. The embedding $A\into\Aq$ corresponds to
  the cone $\cC_{Gz}$. In particular, $\Aq$ contains a unique
  $T$-fixed point, denoted by $0$, such that $\phi_0(0)=z$.

  From this we see that $z$ lies in the source of $\lambda$ on $X$ if
  and only if $\lambda$ has a source in $P_u\times\Aq$. This is
  automatically the case for $P_u$ since $\lambda$ is dominant and
  acts by conjugation. The fixed point set is $P_u^\lambda=P_u\cap
  G^\lambda$. On the other hand, we have
  \begin{equation}
    f_\chi(\lambda(t)a)=t^{-\lambda^r(\chi)}f_\chi(a).
  \end{equation}
  For the limit when $t\pfeil0$ to exist for all $a\in\Aq$ it is
  necessary and sufficient that $-\lambda^r(\chi)\ge0$ for all
  $\chi\in\Xi_\QQ(\Aq)$ with $f_\chi\in\cO(\Aq)$. This condition boils
  down to $-\lambda^r\in\cC_{Gz}$. In that case, one readily checks
  that the fixed point set $\Aq^\lambda$ is the closure of the orbit
  corresponding to the face $\cC(\lambda)$ of $\cC_{Gz}$. The
  restricted morphism
  \begin{equation}\label{eq:41}
    \phi^\lambda:P_u^\lambda\times\Aq^\lambda\pfeil X^\lambda
  \end{equation}
  describes the local structure of $X^\lambda$ in a neighborhood of
  $z$.

  From this we already infer that
  $\Xi_\QQ(X^\lambda)=\Xi_\QQ(\Aq^\lambda)=V(\lambda)^\perp$. We claim
  that $\Xi(X^\lambda)=\Xi_\QQ(X^\lambda)\cap\Xi(X)$ which then proves
  the assertion $\Xi(X^\lambda)=\Xi(X)\cap V(\lambda)^\perp$.  In
  fact, only ``$\supseteq$'' is an issue. To prove it let
  $\chi\in\Xi_\QQ(X^\lambda)\cap\Xi(X)$. Then there are $n\in\ZZ_{>0}$
  and rational semiinvariants $f_\chi$ on $X$ and $f_{n\chi}$ on
  $X^\lambda$ such that $f_\chi^n=\pi_\lambda^*f_{n\chi}$. Let
  $X'\subseteq X^\lambda$ be the open subset on which $f_{n\chi}$ is
  regular. Then the normality of $X$ implies that $f_\chi$ is regular
  on $\pi_\lambda^{-1}(X')$. Since $f_\chi$ is also
  $\lambda$-invariant we conclude that $f_\chi$ pushes down to a
  rational function on $X^\lambda$ which shows
  $\chi\in\Xi(X^\lambda)$, as claimed.  The equality
  $N_\QQ(X^\lambda)=N_\QQ(X)/V(\lambda)$ follows immediately.

  Finally, we compute the fan $\cF(X^\lambda)$. Clearly, it suffices
  to determine its cones $\cC$ of maximal dimension corresponding to
  closed orbits. \autoref{lem:ClosedOrbit} and the discussion above
  show that the closed $G^\lambda$-orbits in $X^\lambda$ correspond
  precisely to those closed $G$-orbits $Gz$ in $X$ such that
  $-\lambda^r\in\cC_{Gz}$. In that case it is easy to check that the
  toroidal embedding $A^\lambda\into\Aq^\lambda$ corresponds to the
  cone $(\cC_{Gz}+\QQ_{\ge0}\lambda)/V(\lambda)$. But these are
  precisely the cones of maximal dimension in $\cF^\lambda$ which
  shows $\cF(X^\lambda)=\cF(X)^\lambda$.
\end{proof}

\begin{corollary}

  Let $\lambda$ and $X$ be as above. Then
  \begin{equation}
    \Sigma(X^\lambda)=\Sigma(X)\cap V(\lambda)^\perp=\Sigma(X)\cap\lambda^\perp.
  \end{equation}

\end{corollary}

\begin{proof}

  The valuation cone $\cV(X^\lambda)$ equals the support of
  $\cF(X)^\lambda$. Its codimension-1-faces are, by construction, the
  codimension-1-faces of $\cV(X)$ which contain $\cC(\lambda)$. From
  $\Xi_p(X^\lambda)=\Xi_p(X)\cap V(\lambda)^\perp$ we get
  $\Sigma(X^\lambda)=\Sigma(X)\cap V(\lambda)^\perp$. The second
  equality follows from the fact that $\<\sigma,\cV(X)\>\ge0$ for all
  $\sigma\in\Sigma(X)$. Hence $\<\sigma,V(\lambda)\>=0$ if and only if
  $\<\sigma,\cC(\lambda)\>=0$ if and only if $\<\sigma,\lambda\>=0$.
\end{proof}

\begin{proposition}\label{prop:GXLambda}

  Let $\lambda$ and $X$ be as above. Then $GX^\lambda=\overline Y_0$
  where $Y_0\subseteq X$ is the $G$--orbit with
  $\cC_{Y_0}=\cC(\lambda)$. Moreover, for any $G$-orbit $Y\subseteq X$
  holds
  \begin{eqnarray}
    \cC(\lambda)\subseteq\cC_Y&\Longleftrightarrow&X^\lambda\cap Y\ne\leer\\
    \cC(\lambda)\supseteq\cC_Y&\Longleftrightarrow&X^\lambda\subseteq\overline Y
  \end{eqnarray}

\end{proposition}

\begin{proof}

  First note that $X^\lambda$ is stable under
  ${}^-\!B={}^-\!B^\lambda{}^-\!G^u_\lambda$. Hence $GX^\lambda$ is
  closed, hence an orbit closure $\overline{Y_0}$, in $X$. Choose a
  closed orbit $Gz$ in $\overline{Y_0}$. The orbits of $X$ which
  contain $Gz$ in their closure correspond precisely to the $T$-orbits
  in the slice $\Aq$. It follows from \eqref{eq:41} that this way
  $\overline{Y_0}$ corresponds to $\Aq^\lambda$ which shows
  $\cC_{Y_0}=\cC(\lambda)$, as claimed.

  The two equivalences follow easily: we have $X^\lambda\cap
  Y\ne\leer$ if and only if $Y\subseteq GX^\lambda=\overline{Y_0}$ if
  an only if $\cC_Y\supseteq\cC(\lambda)$ and
  $X^\lambda\subseteq\overline Y$ if and only if
  $\overline{Y_0}=GX^\lambda\subseteq\overline Y$ if and only if
  $\cC(\lambda)\supseteq\cC_Y$.
\end{proof}

Next we compute the colors of $X^\lambda$. Let $D\subseteq X^\lambda$
be a color and let $D_0$ be its restriction to the open
$G^\lambda$-orbit. Then $\pi_\lambda^{-1}(D_0)$ is irreducible by
\autoref{lem:GenFiber}. Hence its closure $D^*\subset X$ is a
$B$-stable prime divisor.

\begin{proposition}\label{prop:ColorsX}

  Let $\lambda$ and $X$ be as above.

  \begin{enumerate}

  \item\label{it:ColorsX.a} Let $\a\in
    S^\lambda:=S(G^\lambda)=S\cap\lambda^\perp$ and let $D$ be a color
    of $X^\lambda$. Then

    \begin{enumerate}

    \item\label{it:ColorsX.i} $\a$ has the same type for $X^\lambda$
      as it has for $X$,

    \item\label{it:ColorsX.ii} $\delta_D$ is the restriction of
      $\delta_{D^*}$ to $\Xi_\QQ(X^\lambda)\subseteq\Xi_\QQ(X)$, and

    \item\label{it:ColorsX.iii} $D$ is moved by $\a$ if and only if
      $D^*$ is moved by $\a$. In that case $q_{D,\a}=q_{D^*,\a}$.

    \end{enumerate}

  \item\label{it:ColorsX.b} The map $D\mapsto D^*$ is a bijection
    between the set of colors of $X^\lambda$ and the set of colors of
    $X$ which are moved by some $\a\in S^\lambda$.

  \end{enumerate}

\end{proposition}

\begin{proof} The first part of \ref{it:ColorsX.iii} follows from
  $P_\a\subseteq G_\lambda$ and the equivariance of
  $\pi_\lambda$. Then \ref{it:ColorsX.b} is an immediate consequence.
  For \ref{it:ColorsX.ii} recall that $\pi_\lambda^{-1}(D_0)$ is even
  a reduced divisor by \autoref{lem:GenFiber}. Thus,
  $v_{D^*}(\pi_\lambda^*f)=v_D(f)$ for all $f\in
  k(X^\lambda)$. Moreover,there is a commutative diagram
  \begin{equation}
    \cxymatrix{P_\a\times^BD^*\ar[r]^(.6){\phi_{D^*,\a}}&X\\
      P_\a\times^B\pi_\lambda^{-1}(D_0)\ar[d]\ar[r]\ar@{_{(}->}[u]&
      X_\lambda\ar[d]^{\pi_\lambda}\ar@{_{(}->}[u]\\
      P_\a\times^BD_0\ar[r]\ar@{_{(}->}[d]&X^\lambda\ar@{=}[d]\\
      P_\a^\lambda\times^{B^\lambda}D\ar[r]\ar[r]^(.6){\phi_{D,\a}^\lambda}&X^\lambda}
  \end{equation}
  where the middle square is cartesian and the injections are open
  embeddings. It follows that $\phi_{D,\a}^\lambda$ and
  $\phi_{D^*,\a}$ have the same inseparable degree, showing the second
  part of \ref{it:ColorsX.iii}. Both morphisms also have the same
  separable degree ($1$ or $2$). Thus, $\a$ is of type \tAA/ for
  $X^\lambda$ if and only if it is of type \tAA/ for $X$, proving part
  of \ref{it:ColorsX.i}. The other types \tP/, \tB/, and \tA/ are
  distinguished by the number of colors ($0$, $1$, and $2$,
  respectively) moved by $\a$. Thus, the rest of \ref{it:ColorsX.i}
  follows from \ref{it:ColorsX.iii}.
\end{proof}

\section{The interrelation of roots and colors}
\label{sec:RootsAndColors}

In practice, the 1-parameter subgroup $\lambda$ has to be chosen
diligently.

\begin{lemma}\label{lem:1PSG}

  Let $X$ be a complete toroidal spherical $G$-variety and
  $S'\subseteq S$. Then there is a 1-parameter subgroup $\lambda$ such
  that

  \begin{enumerate}

  \item\label{it:1PSG.a} $S(G^\lambda)=S'$ and

  \item\label{it:1PSG.b} the connected center of $G^\lambda$ acts
    trivially on $X^\lambda$.

  \end{enumerate}

\end{lemma}

\begin{proof}

  Let $F\subseteq N_\QQ(T):=\|Hom|(\Xi(T),\QQ)$ be the open face of
  the Weyl chamber defined by $\a=0$ for all $\a\in S'$ and $\a>0$ for
  $\a\in S\setminus S'$. Then $\lambda\in F$ is equivalent to
  \ref{it:1PSG.a} (such $\lambda$ are called \emph{adapted} to $S'$).

  Now consider the projection $\pi:N_\QQ(T)\auf N_\QQ(X)$. Since
  $\pi(F)\subseteq\cV(X)$, the fan $\cF$ associated to $X$ induces a
  complete fan $\cF'$ on $\pi(F)$. Let
  $F^0:=F\setminus\bigcup\pi^{-1}(\cC)$ where $\cC$ runs through all
  $\cC\in\cF'$ with $\|dim|\cC<\|dim|\pi(F)$ which is obviously a
  dense open subset of $F$. We claim that $\lambda\in F^0$ ensures the
  second property \ref{it:1PSG.b}.

  Indeed, $\pi(\lambda)$ lies by construction in the relative interior
  of a maximal dimensional cone $\cC'$ of $\cF'$. This implies that
  $\pi(F)\subseteq V'$ where $V'$ is the span of $\cC'$. Now let
  $\cC\in\cF$ be minimal with $\cC'\subseteq\cC$. Then $\pi(\lambda)$
  is also in the relative interior of $\cC$. Thus, the subspace $V$
  spanned by $\cC$ contains $\pi(F)$. Hence
  $\<F\>_\QQ\subseteq\pi^{-1}(V)$ which implies \ref{it:1PSG.b}.
\end{proof}

If the fan is changed one can do better.

\begin{lemma}\label{lem:1PSG2}

  Let $X_0$ be a homogeneous spherical $G$-variety and $S'\subseteq
  S$. Then there is a $1$-parameter subgroup $\lambda$ and a toroidal
  compactification $X$ of $X_0$ such that

  \begin{enumerate}

  \item\label{it:1PSG2.a} $S(G^\lambda)=S'$ and

  \item\label{it:1PSG2.b}
    $\Xi_\QQ(X^\lambda)=\<\Sigma(X^\lambda)\>_\QQ$.

  \end{enumerate}

\end{lemma}

\begin{proof}

  Same construction as above but this time we choose $\cF$ such that
  $\|dim|\cC$ is as large as possible, i.e., the dimension of the
  smallest face of $\cV$ which contains $\cC'$. This means precisely
  \ref{it:1PSG2.b}.
\end{proof}

\begin{remark}

  A rather trivial application of the last lemma is when $S'=S$. Then
  $X^\lambda$ has the same roots and colors as $X_0$ but
  $\Xi_\QQ(X^\lambda)$ is spanned by $\Sigma(X)$.

\end{remark}

Following Luna~\cite{LuGC}, an important application of this technique
is:

\begin{proposition}\label{prop:RootColors}

  Let $\a\in S$ be a simple root.

  \noindent If $p\ne2$ then

  \begin{enumerate}

  \item\label{it:RootColors.a} $\a\in\Sigma(X)$ if and only if $\a$ is
    of type \tA/. Thus $S^\tA/=S\cap\Sigma(X)$.

  \item\label{it:RootColors.b} $2\a\in\Sigma(X)$ if and only if $\a$
    is of type \tAA/. Thus
    $S^\tAA/=S\cap{\textstyle{\frac12}}\Sigma(X)$.

  \end{enumerate}

  If $p=2$ then

  \begin{enumerate}[resume]

  \item\label{it:RootColors.c} $\a\in\Sigma(X)$ if and only if $\a$ is
    of type \tA/ or \tAA/. Thus $S^\tA/\cup S^\tAA/=S\cap\Sigma(X)$.

  \end{enumerate}

\end{proposition}

\begin{proof}

  Without loss of generality we may replace $X$ by a toroidal
  compactification of its open $G$-orbit. The corresponding fan is
  denoted by $\cF$. Choose $\lambda$ as in \autoref{lem:1PSG} with
  $S':=\{\a\}$. Then $G^\lambda$ acts on $X^\lambda$ only via a
  semisimple quotient $G_0$ of rank $1$. Let $G_0/H_0$ be the open
  $G_0$-orbit in $X^\lambda$.

  Assume first $p\ne2$. Then $\a\in S^\tA/(X)$ iff $\a\in
  S^\tA/(X^\lambda)$ iff $H_0^{\|red|}\sim T_0$ iff
  $\a\in\Sigma(X^\lambda)$ iff $\a\in\Sigma(X)$. This proves
  \ref{it:RootColors.a}.

  The argument for \ref{it:RootColors.b} is analogous with $T_0$
  replaced by $N(T_0)$. Finally, for $p=2$ we argue with
  $\Sigma(G_0/T_0)=\Sigma(G_0/N(T_0))=\{\a\}$.
\end{proof}

\begin{remarks}

  1. The proposition shows that in case $p\ne2$ the type of the simple
  roots can be recovered from $S^\tP/$ and $\Sigma(X)$ as follows:
  \begin{equation}
    \text{$\a\in S$ is of type }
    \begin{cases}
      \tP/&\text{if $\a\in S^\tP/$}\\
      \tA/&\text{if $\a\in\Sigma(X)$}\\
      \tAA/&\text{if $2\a\in\Sigma(X)$}\\
      \tB/&\text{otherwise}
    \end{cases}
  \end{equation}
  This way, all colors can be recovered but some might appear multiple
  times (see \autoref{lem:Multiple}). For colors of type \tB/, that
  behavior is controlled by $\Sigma(X)$, as well (see the following
  \autoref{prop:TypeB}).

  2. For $p=2$ and $\a\in S^\tAA/$ it is tempting to define the
  corresponding spherical root to be $2\a$ instead of $\a$. This would
  make parts \ref{it:RootColors.a} and \ref{it:RootColors.b} of
  \autoref{prop:RootColors} work uniformly in all characteristics. We
  opted against this procedure. The main reason is that otherwise
  spherical roots would not be roots (possibly not simple) of some
  root system. Example: For $p=2$ and $X=\SL(3)/\SO(3)$ the two roots
  are $\a_1$ and $\a_1+\a_2$ (see \cite{Schalke}) which are visibly
  contained in an $\A_2$-root system. On the other hand, the root
  $\a_1$ is of type \tAA/ and the set $\{2\a_1,\a_1+\a_2\}$ is not
  part of any root system.

\end{remarks}

\begin{proposition}\label{prop:TypeB}.

  For two distinct simple roots $\a_1,\a_2\in S^\tB/$ are equivalent:

  \begin{enumerate}

  \item\label{it:TypeB.a} $\a_1$ and $\a_2$ move the same color $D$.

  \item\label{it:TypeB.b}$\a_1$ and $\a_2$ are orthogonal to each
    other and $q_1\a_1+q_2\a_2\in\Sigma(X)$ for two $p$-powers
    $q_1,q_2$ (one of which is necessarily equal to $1$).

  \end{enumerate}

  \noindent In this case

  \begin{equation}\label{eq:37}
    q_1^{-1}\a_1^r=q_2^{-1}\a_2^r.
  \end{equation}

  Moreover, $D$ is not moved by any other simple root.

\end{proposition}

\begin{proof}

  Again replace $X$ by a toroidal compactification and choose
  $\lambda$ as in \autoref{lem:1PSG} with $S'=\{\a_1,\a_2\}$. Now
  $G^\lambda$ acts on $X^\lambda$ via a semisimple group $G_0$ of rank
  $2$. The simple roots of $G_0$ are $\a_1$ and $\a_2$, and both are
  of type $\tB/$ with respect to $X^\lambda$.

  Assume first \ref{it:TypeB.b}. Then $G_0$ is of type $\A_1\A_1$. An
  easy inspection of its subgroups shows that $X^\lambda$ has a
  spherical root of the given form if and only if its open orbit is
  isogenous to $\SL(2)\times\SL(2)/(F_{q_1}\times F_{q_2})\SL(2)$
  (with $F_q$=Frobenius morphism of $\SL(2)$). In that case one checks
  that \ref{it:TypeB.a} holds for $X^\lambda$ and therefore for $X$.

  Conversely assume \ref{it:TypeB.a}. Then $X^\lambda$ has precisely
  one color $D$ which is moved by both simple roots. Thus,
  \eqref{eq:5} implies that a relation like \eqref{eq:37} holds. In
  particular, the rank of $X^\lambda$ is~$1$.

  Now one could use the classification of spherical varieties of rank
  $\le1$ in \cite{SpRoot} and conclude that $G_0$ is of type
  $\A_1\A_1$ having a root of the given form. A selfcontained argument
  goes as follows. We may assume that $X=G/H$ is homogeneous where $H$
  is reduced and connected. The color and the half{}line $\cV(X)$ lie
  opposite to each other. By \cite{Hyd}*{Thm.~6.7} the variety $X$ is
  affine, thus $H$ is reductive. Since there is only one color,
  formula \eqref{eq:40} implies that $H$ is semisimple. Moreover, the
  dimension formula \cite{Hyd}*{Thm.~6.6} shows that $\|dim|H=3,4,5,7$
  for $G=\A_1\A_1,\A_2,\B_2, \G_2$, respectively. This shows that $G$
  is isogenous to $\SL(2)\times\SL(2)$ and that $H\cong\SL(2)$ is
  embedded diagonally using the Frobenius morphisms. The assertion
  \ref{it:TypeB.b} follows.

  Formula \eqref{eq:37} follows immediately from
  \eqref{eq:5}. Finally, assume $\a_3\in S$ moves $D$, as well. Then
  $\a_3\in S^\tB/$ (\autoref{lem:Multiple}). Thus, by the above,
  $\a_3$ would be orthogonal to $\a_1$ and $\a_2$. Moreover,
  $q_1'\a_1+q_3\a_3\in\Sigma(X)$ for some $p$-powers $q_1'$,
  $q_3$. Now equation \eqref{eq:37} implies the contradiction
  \begin{equation}
    2q_1'q_1^{-1}=\<q_1'\a_1+q_3\a_3,q_1^{-1}\a_1^\vee\>=
    \<q_1'\a_1+q_3\a_3,q_2^{-1}\a_2^\vee\>=0.
  \end{equation}
\end{proof}

\section{Localization at $\Sigma$}
\label{sec:SecondLoc}

Localization at $S$ allows one to pass from $S$ to a subset.  There is
a second, older, kind of localization which does the same thing with
$\Sigma(X)$. Geometrically, it simply corresponds to looking at an
orbit in a carefully chosen toroidal embedding. The next result
summarizes what was already known about localization at $\Sigma$:

\begin{proposition}\label{prop:LocalizationAtSigma}

  Let $X$ be a toroidal spherical variety and let $Y\subseteq X$ be an
  orbit. Put $V:=\<\cC_Y\>^\perp\subseteq\Xi_\QQ(X)$. Then

  \begin{enumerate}

  \item\label{it:LocalizationAtSigma.a} $\Xi_p(Y)=\Xi_p(X)\cap V$.
  \item\label{it:LocalizationAtSigma.b} $\Sigma(Y)=\Sigma(X)\cap V$.
  \item\label{it:LocalizationAtSigma.c} $S^\tP/(Y)=S^\tP/(X)$.

  \end{enumerate}

\end{proposition}

\begin{proof}

  Part \ref{it:LocalizationAtSigma.a} follows e.g. from
  \cite{Hyd}*{Thm.~1.3}. Moreover, $\cV(Y)=(\cV(X)+V)/V$ where
  $V=\<\cC\>_\QQ$ (follows from \cite{InvBew}*{Satz~7.4}) which
  implies \ref{it:LocalizationAtSigma.b}. Part
  \ref{it:LocalizationAtSigma.c} follows for example from the fact
  that all closed orbits in any toroidal compactification of $X$ are
  of the form $G/Q$ with $Q^{\|red|}={}^-\!P$ where $P$ is the
  parabolic corresponding to $S^\tP/$.
\end{proof}

If $p\ne2$, then the remark after \autoref{prop:RootColors} allows now
to determine the type of a simple root for $Y$.
\begin{eqnarray}
  S^\tP/(Y)&=&S^\tP/(X)\\
  \label{eq:43}S^\tA/(Y)&=&S^\tA/(X)\cap V\\
  \label{eq:44}S^\tAA/(Y)&=&S^\tAA/(X)\cap V\\
  S^\tB/(Y)&=&S^\tB/(X)\cup (S^\tA/(X)\setminus V)\cup (S^\tAA/(X)\setminus V)
\end{eqnarray}
For $p=2$ equations \eqref{eq:43} and \eqref{eq:44} have to be
replaced by the weaker equality
\begin{equation}
  S^\tA/(Y)\cup S^\tAA/(Y)=(S^\tA/(X)\cap V)\cup(S^\tAA/(X)\cap V)
\end{equation}
The next lemma shows (in particular) that moreover
\begin{equation}
  S^\tAA/(Y)\subseteq S^\tAA/(X)\cap V
\end{equation}

\begin{lemma}\label{lem:LocSigma2}

  Let $X$ and $Y$ be as above and let $\a\in S\cap V$ be of type \tA/
  for $X$. Then $\a$ is also of type \tA/ for $Y$. Moreover, let $D$
  be a color of $X$ which is moved by $\a$. Then $E=(D\cap
  Y)^{\|red|}$ is a color of $Y$ which is moved by $\a$ and there is a
  $p$-power $q$ such that $\delta_E$ is the restriction of $q\delta_D$
  to $\Xi_p(Y)$.

\end{lemma}

\begin{proof}

  We plan to use localization at $S$ but face the problem that $\cC_Y$
  may not meet $N^-_\QQ(X)$, the image of the antidominant Weyl
  chamber in $N_\QQ(X)$. To bypass this problem, we go one dimension
  up: the group $\Gq:=G\times\Gm$ acts on $\Xq^0:=X\times\Gm$. Then
  $N_\QQ(\Xq^0)=N_\QQ(X)\oplus\QQ$ and
  $\cV(\Xq^0)=\cV(X)\times\QQ$. Now choose any $v_0$ in the relative
  interior of $N^-(X)\cap\{\a=0\}\subseteq\cV(X)$ and put
  $\overline\cC:=(\cC\times0)+\QQ_{\ge0}(v_0,1)$. Choose any fan $\cF$
  whose support is $\cV(\Xq^0)$ and which contains
  $\overline\cC$. This gives rise to a toroidal $\Gq$-variety
  $\Xq$. Moreover, $\Xq$ contains $Y\times\Gm$ since $\cC_Y$ is a face
  of $\overline\cC$. Its closure is denoted by $\Yq$.

  Now choose any $v\in\cC^0$ small enough such that $v+v_0$ is still
  in the relative interior of $N^-(X)\cap\{\a=0\}$ and choose
  $a\in\ZZ_{>0}$ such that $v_1:=a(v+v_0,1)$ is the image of a
  1-parameter subgroup $\lambda$ of $\Gq$. Then, by construction,
  $S^\lambda=\{\a\}$ and $\Gq^\lambda$ is of semisimple rank
  $1$. Moreover, $\Xq^\lambda$ is contained in $\Yq$ by
  \autoref{prop:GXLambda}.  Thus, we get a diagram
  \begin{equation}
    \cxymatrix{
      \tilde Y_\lambda\ar@{>>}[r]\ar[d]^{\tilde\pi_\lambda}&
      \Yq_\lambda\ar@{^{(}->}[r]\ar[d]^{\pi_\lambda}&\Xq_\lambda\ar[d]^{\pi_\lambda}\\
      \tilde Y^\lambda\ar@{>>}[r]^\nu&\Xq^\lambda\ar@{=}[r]&\Xq^\lambda\\}
  \end{equation}
  where $\tilde Y$ is the normalization of $\Yq$ and where the
  vertical arrows represent Bia\l ynicki-Birula contractions on the
  open cell. By \autoref{prop:ColorsX}, the type of $\a$ on
  $\Xq^\lambda$ is \tA/. This means that the open $\Gq^\lambda$-orbit
  in $\Xq^\lambda$ is of the form $\Gq^\lambda/H_0$ where
  $H_0^{\|red|}$ is diagonalizable.

  Now we argue that $\nu$ is purely inseparable. Indeed, the open
  orbit in $\tilde Y^\lambda$ is $\Gq^\lambda/H_1$ with
  $H_1^{\|red|}\subseteq H_0^{\|red|}\subseteq T$. This already shows
  that $\a$ is of type \tA/ for $\tilde Y$, hence for $\Yq$ and
  $Y$. Since both $H_1^{\|red|}$ and $H_0^{\|red|}$ are linearly
  reductive abelian groups we have
  \begin{equation} [H_0^{\|red|}:H_1^{\|red|}]=[\Xi_p(\tilde
    Y^\lambda):\Xi_p(\Xq^\lambda)].
  \end{equation}
  On the other hand
  \begin{equation}
    \Xi_p(\tilde Y^\lambda)=\Xi_p(\tilde Y)=
    \Xi_p(\Yq)=\Xi_p(\Xq)\cap\Vq=\Xi_p(\Xq^\lambda)
  \end{equation}
  where $\Vq:=\<\overline\cC\>$. This shows that $\nu$ is generically
  injective and therefore purely inseparable.

  The color $D$ gives rise to the color $D\times\Gm$ of $\Xq$. Its
  image $D_0$ in $\Xq^\lambda$ is a color of $\Xq^\lambda$. Now
  $E_0=\nu^{-1}(D_0)^{\|red|}$ is a color of $\tilde
  Y^\lambda$. Clearly $\nu^{-1}(D_0)=qE_0$ for some $p$-power
  $q$. Finally, the closure of $(\tilde\pi_\lambda)^{-1}(E_0)$ is a
  color of $\Yq$ which is of the form $E\times\Gm$ where $E=(D\cap
  Y)^{\|red|}$.

  Recall $V=\<\cC\>_\QQ^\perp$. Then
  \begin{equation}
    \Vq=\{(\chi,-v_0(\chi)\mid\chi\in V\}\cong V
  \end{equation}
  and $\Xi_\QQ(\Xq)=\Xi_\QQ(X)\oplus\QQ$, $\Xi_\QQ(\Yq)=V\oplus \QQ$,
  and $\Xi_\QQ(\Xq^\lambda)$. Thus for $\chi\in \Xi_\QQ(X)$:
  \begin{equation}
    \delta_D(\chi)=\delta_{D\times\Gm}(\chi,0)=
    \delta_{D\times\Gm}(\chi,-v_0(\chi))=\delta_{D_0}(\chi,-v_0(\chi))
  \end{equation}
  and similarly $\delta_E(\chi)=\delta_{E_0}(\chi,-v_0(\chi))$ for all
  $\chi\in V$. Thus for $\chi\in V$:
  \begin{equation}
    \delta_D(\chi)=\delta_{D_0}(\chi,-v_0(\chi))=
    q\delta_{E_0}(\chi,-v_0(\chi))=q\delta_E(\chi).
  \end{equation}
\end{proof}

\begin{remark}

  With these results it is possible to recover all colors of $Y$ and
  which color is being moved by which root. In characteristic zero
  this is good enough to compute the entire spherical system of
  $Y$. In positive characteristic we are missing information on the
  degrees $q_{D,\a}$ of $Y$, though. We plan to return to this
  question in the future.

\end{remark}

Localization at $\Sigma$ is, a priori, not possible for all subsets of
$\Sigma(X)$. Therefore, we define:

\begin{definition}

  A subset $\Sigma'$ of $\Sigma(X)$ is called a \emph{set of
    neighbors} if there is $v\in\cV(X)$ such that
  \begin{equation}
    \Sigma'=\{\sigma\in\Sigma(X)\mid v(\sigma)=0\}.
  \end{equation}
  Equivalently, $\Sigma'$ is a set of neighbors if $\QQ_{\ge0}\Sigma'$
  is a face of $\QQ_{\ge0}\Sigma(X)$.  Two spherical roots $\a$ and
  $\beta$ are called \emph{neighbors} if they are distinct and if
  $\{\a,\b\}$ is a set of neighbors.

\end{definition}

Clearly, if $\Sigma(X)$ is linearly independent then all subsets are
sets of neighbors. This is always the case if $p\ne2$ (see
\cite{BriES} for $\|char|k=0$ and \cite{SpRoot}*{\Neighbor} for the
general case). For $p=2$ and $X=SL(4)/SO(4)$, Schalke has shown
(unpublished) that
$\Sigma(X)=\{\a_1,\a_1+\a_2,\a_2+\a_3,\a_3\}$. Since
$\a_1+(\a_2+\a_3)=(\a_1+\a_2)+\a_3$, the pairs $(\a_1,\a_2+\a_3)$ and
$(\a_1+\a_2,\a_3)$ are not neighbors. In fact, $\cV(X)$ is the cone
over a quadrangle and the given pairs correspond to opposite faces.

\begin{lemma}\label{lem:NeighborSimple}

  If $\a,\b\in\Sigma(X)$ are multiples of simple roots then they are
  neighbors.

\end{lemma}

\begin{proof}

  The set $\QQ_{\ge0}\a+\QQ_{\ge0}\b$ is already a face of
  $\QQ_{\ge0}S$ and therefore also for the smaller cone
  $\QQ_{\ge0}\Sigma(X)$.
\end{proof}

The following statement can be used to exclude certain configurations
of colors and roots (see \cite{SpRoot}*{Proof of Thm.~4.5}).

\begin{proposition}\label{prop:typeAcolorInequal}

  Let $D\in\Delta(X)^\tA/$ be moved by $\a\in S^\tA/$ and let
  $\sigma\in\Sigma(X)$ be a neighbor of $\a$ with
  $\delta_D(\sigma)>0$. Then $\sigma\in S^\tA/$ and $D$ is moved by
  $\sigma$, as well.

\end{proposition}

\begin{proof}

  We first reduce to the case that $X$ is of rank $2$ with
  $\Sigma=\{\a,\sigma\}$. Because $\a$ and $\sigma$ are neighbors one
  can choose a pointed cone $\cC$ inside $\cV(X)\cap\{\a=\sigma=0\}$
  which is of codimension $2$ in $N_\QQ(X)$. Let $\Xq$ be the simple
  embedding corresponding to $\cC$ and let $Y$ be its closed
  orbit. Then \autoref{prop:LocalizationAtSigma} implies
  $\Sigma(Y)=\{\a,\sigma\}$.  Moreover, using the remark after
  \autoref{lem:1PSG2} there is a spherical variety $Y'$ with
  $\Sigma(Y')=\{\a,\sigma\}$ and $\|rk|Y'=2$. Moreover,
  \autoref{lem:LocSigma2} implies that $Y'$ still has a color $E$
  moved by $\a$ with $\delta_E(\sigma)=q\delta_D(\sigma)>0$. Let us
  assume that we can prove that $\sigma\in S^\tA/(Y')$ and that $E$ is
  moved by $\sigma$. Clearly, $D$ is moved by $\sigma$ in $X$, as
  well.  Moreover, since $\a\in\Sigma(X)$, either $\a\in S^\tA/(X)$ or
  $p=2$ and $\a\in S^\tAA/(X)$. But the latter case cannot happen
  since then $D$ could not be moved by any other simple root
  (\autoref{lem:Multiple}). So $\sigma\in S^\tA/(X)$, which finishes
  the reduction step.

  From now on we assume $\|rk|X=2$ and, without loss of generality,
  that $X=G/H$ where $H$ is reduced. Since, $\delta_D(\sigma)>0$ by
  assumption and $\delta_D(\a)=q_{D,\a}>0$, the cone generated by
  $\cV(X)$ and $\delta_D$ is the entire space $N:=N_\QQ(X)$. From that
  we get a morphism $X=G/H\pfeil G/P$ with $\|rank|G/P=\|dim|N/N=0$
  (\cite{Hyd}*{4.6})\footnote{The idea for this construction is due to
    Guido Pezzini.}. Hence $P$ is a parabolic subgroup with an
  identification $\Delta(G/P)=\Delta(G/H)\setminus\{D\}$. We may
  choose $P$ in such a way that it is \emph{opposite} to $B$.

  Every $\b\in S\setminus S^\tP/$ moves at least one color and $\a$
  moves even two. Assume first that these colors are \emph{not\/} all
  different. Then, according to \autoref{lem:Multiple} and
  \autoref{prop:TypeB} there are two possibilities:

  {\it a)} $\sigma=\a_1+q\a_2$ with $\a_1,\a_2\in S$ orthogonal. But
  then $\a_1$ and $\a_2$ would move the same color in $G/P$ which is
  impossible.

  {\it b)} $S^\tA/$ contains another element $\b$ besides $\a$. But
  then $\b\in\Sigma(X)$ (\autoref{prop:RootColors}), hence
  $\sigma=\b\in S^\tA/$. Moreover, there is a color $D'$ moved by both
  $\a$ and $\sigma$. We claim that $D'=D$. Suppose not. Then $D$ and
  $D'$ are the two colors moved by $\a$ and
  \autoref{prop:ColorRelations} implies the contradiction
  \begin{equation}
    \delta_D^{(\a)}(\sigma)=\<\sigma,\a^\vee\>-\delta_{D'}^{(\a)}(\sigma)=
    \<\sigma,\a^\vee\>-\frac{q_{D',\a}}{q_{D',\sigma}}\delta_{D'}^{(\sigma)}(\sigma)=
    \<\sigma,\a^\vee\>-\frac{q_{D',\a}}{q_{D',\sigma}}<0.
  \end{equation}
  Thus, we are exactly in the asserted situation, i.e., $\sigma\in
  S^\tA/$ moving $D$.

  So, assume from now on that the colors moved by all the $\b\in
  S\setminus S^\tP/$ are different. Then
  \begin{equation}\label{eq:38}
    \#\Delta(G/P)\ge\#S\setminus S^\tP/.
  \end{equation}
  Consider a toroidal completion $\Xq$ of $X$. Then the morphism
  $X\pfeil G/P$ extends to $\Xq$ (\cite{Hyd}*{4.1}). Every closed
  orbit in $\Xq$ is isogenous to $G/{}^-\!P_X$ where $P_X$ is the
  parabolic attached to $S^\tP/$. Hence $P_X\subseteq
  P^{\|red|}$. Thus \eqref{eq:38} implies $P=P_X$.

  Let $Y=G/H_1\subset\Xq$ be the rank-1-orbit corresponding to the
  half{}line $\cV(X)\cap\{\sigma=0\}$. Then
  $\Sigma(Y)=\{\sigma\}$. Because of $P=P_X$, the fiber
  $P_X/H_1^{\|red|}$ is one-dimensional and therefore isomorphic to
  ${\bf P}^1$, $\Gm$, or ${\bf A}^1$.  The first case is impossible
  since $H_1$ is not parabolic. The second case is excluded since
  $H_1$ is not horospherical. Thus $P_X/H_1^{\|red|}\cong{\bf A}^1$.

  This means in particular that (a conjugate of) $H_1$ contains the
  maximal torus $T$ of $G$ and that $H_1$ contains all root subgroups
  $U_\b$ which are contained in $P_X$ except for one, which is denoted
  by $\gamma$, and which lies in $P_X^u$. The $U_\b$ corresponding to
  $\b\in S$ generate the maximal unipotent subgroup of $G$. This
  implies $\gamma\in S$. Moreover $U_\gamma$ is a 1-dimensional module
  for the Levi part of $P_X$. This shows that $H_1^{\|red|}$ is, in
  fact, induced from $\PGL(2)/\Gm$ (on induction in arbitrary
  characteristic see \cite{SpRoot}*{\S2, in particular Lemma
    2.1}). Hence $\sigma\in S^\tA/$. But in that case \eqref{eq:38}
  would be a strict inequality which is not true because of $P=P_X$.
\end{proof}

\autoref{prop:typeAcolorInequal} can be used to give bounds for
$\delta_D(\sigma)$:

\begin{corollary}\label{cor:bound}

  With $\a$ and $\sigma$ as above, assume that $\sigma\not\in S$ or
  that $\sigma\in S$ but does not move either color moved by
  $\a$. Then
  \begin{equation}\label{eq:42}
    q_{D,\a}^{-1}\<\sigma,\a^\vee\>\le\delta_D(\sigma)\le0.
  \end{equation}

\end{corollary}

\begin{proof}

  The righthand inequality follows directly from
  \autoref{prop:typeAcolorInequal}. For the lefthand inequality apply
  \autoref{prop:typeAcolorInequal} to the other color $D'$ moved by
  $\a$ and observe that $\delta_{D'}^{(\a)}=\a^r-\delta_D^{(\a)}$
  (\autoref{prop:ColorRelations}).
\end{proof}

In positive characteristic, these bounds are less valuable since we
didn't derive a bound on the denominator of $\delta_D(\sigma)$. Such a
bound exists (in terms of the $q_{D,\a}$'s) and will be included in a
future paper. Then \eqref{eq:42} leaves only finitely many
possibilities for $\delta_D(\sigma)$.

\section{The $p$-spherical system}
\label{sec:SphericalSystem}

We summarize what we have proved so far in terms of a combinatorial
structure which generalizes Luna's spherical systems. But first, we
need some more terminology:

\begin{definition} Let $G$ be a connected reductive group.

  \begin{enumerate}

  \item An element $\sigma\in\Xi_\QQ(T)$ is called a \emph{spherical
      root for $G$} if there is a spherical $G$-variety $X$ such that
    $\sigma$ is a spherical root for $X$. The set of spherical roots
    for $G$ is denoted by $\Sigma(G)$.

  \item A spherical root $\sigma\in\Sigma(G)$ is \emph{compatible} to
    a subset $S^\tP/\subseteq S$ if there is a spherical $G$-variety
    $X$ with $\sigma\in\Sigma(X)$ and $S^\tP/=S^\tP/(X)$.

  \end{enumerate}

\end{definition}

\begin{remarks} 1. \autoref{prop:LocalizationAtSigma} shows that in
  the definition above one may assume without loss of generality
  $\|rk|X=1$.

  2. Spherical varieties of rank $1$ have been classified by Akhiezer
  \cite{Akh} in characteristic zero and \cite{SpRoot} in general. In
  particular, for every $G$ there is a complete description of
  $\Sigma(G)$ (see \cite{SpRoot}*{\S\RankOne\ and \S\TABLE}).

  3. One result of that classification is that $\Sigma(G)$ is infinite
  unless $\|char|k=0$ or $G$ is simple of rank $\le2$.

\end{remarks}

For the following recall that $\ZZ_p=\ZZ[\frac1p]$ and
$\Xi_p:=\Xi\otimes\ZZ_p$ for any abelian group $\Xi$.

\begin{definition} Let $p\ne2$. Then a $p$-spherical system for $G$
  consists of

  \begin{itemize}

  \item a subgroup $\Xi\subseteq\Xi(T)$,
  \item a subset $\Sigma\subseteq\Xi_p\cap\Sigma(G)$,
  \item a subset $S^{\tP/}\subseteq S$,
  \item a finite set $\Delta^\tA/$,
  \item a map
    $\delta:\Delta^\tA/\pfeil\|Hom|(\Xi,\ZZ):D\mapsto\delta_D$, and
  \item a map $S\setminus (S^\tP/\cup S^\tA/)\pfeil p^\NN:\a\mapsto
    q_\a$ where $S^\tA/:=S\cap\Sigma$.

  \end{itemize}

\end{definition}

\noindent Of course, these data are subject to some conditions. Here,
we list only those which are straightforward generalizations of Luna's
axioms. It is safe to say that more axioms have to be imposed which
deal specifically with issues of positive characteristic. We keep the
notation that $\a^r$ denotes the restriction of $\a^\vee$ to $\Xi$.

\begin{itemize}

\item[A1]All $\sigma\in\Sigma$ are primitive vectors of $\ZZ
  S\cap\Xi_p$.

\item[A2]$\a^r=0$ for all $\a\in S^\tP/$.

\item[A3]Every $\sigma\in\Sigma$ is compatible to $S^\tP/$.

\item[A4]For all $D\in\Delta^\tA/$ and $\sigma\in\Sigma\setminus
  S^\tA/$ holds $\delta_D(\sigma)\le0$.

\item[A5] For every $\a\in S^\tA/$ there are exactly two
  $D\in\Delta^\tA/$ with $\delta_D(\a)>0$.  Conversely, for every
  $D\in\Delta^\tA/$ there is at least one $\a\in S^\tA/$ with
  $\delta_D(\a)>0$.

\item[A6]For $\a\in S^\tA/$ let $D^+\ne D^-\in\Delta^\tA/$ with
  $\delta_{D^\pm}(\a)>0$. Then
  $q_{\a,D^\pm}:=\delta_{D^\pm}(\a)^{-1}\in p^\NN$ and
  $q_{\a,D^+}\delta_{D^+}+q_{\a,D^-}\delta_{D^-}=\a^r$.

\item[A7]Let $\a\in S$ with $2\a\in\Sigma$. Then $\a\not\in\Xi_p$ and
  $\frac12\a^r(\Xi_p)\subseteq\ZZ_p$. Moreover, $\a^r(\sigma)\le0$ for
  all $\sigma\in\Sigma\setminus\{2\a\}$.

\item[A8]Let $q\a_1+\a_2\in\Sigma$ with $\a_1\perp\a_2$. Then
  $q^{-1}\a_1^r=\a_2^r$ and $q^{-1}q_{\a_1}=q_{\a_2}$.

\end{itemize}

The point is, of course, that for $p\ne2$ every homogeneous spherical
variety $X$ gives rise to a $p$-spherical system. More specifically we
put
\begin{equation}
  \Xi:=\Xi(X),\ \Sigma:=\Sigma(X),\ S^\tP/:=S^\tP/(X),\
  \Delta^\tA/:=\Delta^\tA/(X),\ \delta_D:=\delta_D^X.
\end{equation}
The only new constituents are the $p$-powers. For $\a\in
S\setminus(S^\tP/\cup S^\tA/)=S^\tB/\cup S^\tAA/$ we define $q_\a$ as
$q_{\a,D}$ from \autoref{prop:ColorRelations} where $D$ is the unique
color moved by $\a$.

Now we verify all axioms.

\begin{itemize}

\item[A1] Holds by definition of $\Sigma(X)$.

\item[A2] See \autoref{prop:ColorRelations}.

\item[A3] Follows from the definition of ``compatibility''.

\item[A4] This is \autoref{cor:bound} in conjunction with
  \cite{SpRoot}*{\Neighbor} which implies that for $p\ne2$ any two
  spherical roots are neighbors.

\item[A5] The first part follows also from \autoref{cor:bound} and
  \autoref{prop:ColorRelations}. The second part holds by definition
  of $\Delta^\tA/(X)$.

\item[A6]This follows from \autoref{prop:ColorRelations}.

\item[A7]The first part is
  \autoref{prop:RootColors}\ref{it:RootColors.b} and
  \autoref{cor:Parity}. The second follows from
  \cite{SpRoot}*{\obtuse}.

\item[A8]This is \autoref{prop:TypeB} and
  \autoref{prop:ColorRelations}.

\end{itemize}

\begin{remarks} 1. In characteristic $0$, Luna (see
  \cite{LuTypA}*{5.1}) used Wasserman's tables \cite{Wass} of
  spherical rank-2-varieties to verify the axioms. So our approach is
  more conceptual in that it uses only the classification of rank-$1$-
  but not of rank-$2$-varieties\footnote{Note, that the tables in
    \cite{Wass} are slightly incomplete: of the series $G
    =\Sp(2n)\times\Sp(2)$, $H=\Sp(2n-2)\times\Sp(2)$ with $n\ge2$ only
    the first case, $n=2$, is stated.  I would like to thank Guido
    Pezzini for pointing that out to me.}.

  2. The case $p=2$ requires some modifications. To distinguish simple
  roots of type \tA/ and \tAA/ we redefine $S^\tA/$ as
  \begin{equation}
    S^\tA/:=\{\a\in S\cap\Sigma\mid
    \delta_D(\a)>0\hbox{ for some }D\in\Delta^\tA/\}
  \end{equation}
  This works indeed for spherical systems coming from spherical
  varieties: Suppose there is $\a\in S^\tAA/(X)$ and $D\in\Delta^\tA/$
  with $\delta_D(\a)>0$. Then $D$ is moved by some $\beta\in
  S^\tA/$. Since $\a$ and $\b$ are neighbors
  (\autoref{lem:NeighborSimple}) we get a contradiction to
  \autoref{prop:typeAcolorInequal}.

  With this change all axioms hold for $p=2$ except for one: in A4 one
  has to require that $\sigma$ and $\a$ are neighbors. Observe that A7
  is vacuously satisfied.

  3. It is a natural question whether spherical varieties are
  classified by their $p$-spherical system. In characteristic zero,
  the answer is ``yes'' according to work by Luna~\cite{LuTypA},
  Losev~\cite{Losev}, Cupit-Foutou~\cite{CuFou}, and
  Bravi-Pezzini~\cite{BrPezz1,BrPezz2,BrPezz3}. For $p\ne2$ or $3$, it
  might be possible that the $p$-spherical system determines the
  variety uniquely. For example, all complete homogeneous varieties
  are classified by $p$-spherical systems with $\Xi=0$ (see the
  example before \autoref{lem:finitemor}). Furthermore, the author
  convinced himself that this also holds for spherical varieties of
  rank 1. If $p=2$ or $p=3$ then uniqueness does not even hold for
  complete homogeneous varieties (see \cite{Wenz}*{Prop.~4}) due to
  exceptional isogenies. If $p=2$ then uniqueness is wrong already for
  $G=\SL(2)$ as then $G$ contains non-standard horospherical subgroup
  schemes (see \cite{KnRk1}).

  4. The above list of axioms A1--A8 is definitely only
  preliminary. Even in the rank-1-case they do not suffice. For
  example, there is no axiom bounding the lattice $\Xi$ from below. We
  plan to return to this problem in the future.

\end{remarks}


\begin{bibdiv}
  \begin{biblist}


    \bib{Akh}{article}{ author={Akhiezer, D. N.}, title={Equivariant
        completion of homogeneous algebraic varieties by homogeneous
        divisors}, journal={Ann. Global Anal. Geom.}, volume={1},
      pages={49--78}, date={1983} }

    \bib{BiaBir}{article}{ author={Bia\l ynicki-Birula, A.},
      title={Some properties of the decompositions of algebraic
        varieties determined by actions of a torus},
      journal={Bull. Acad. Polon. Sci. Sér. Sci. Math. Astronom. Phys.},
      volume={24}, pages={667--674}, date={1976} }

    \bib{BiaSwi}{article}{ author={Bia{\l}ynicki-Birula, A.},
      author={\polS wi\pole cicka, J.}, title={Complete quotients by
        algebraic torus actions}, conference={ title={Group actions
          and vector fields}, address={Vancouver, B.C.}, date={1981},
      }, book={ series={Lecture Notes in Math.}, volume={956},
        publisher={Springer}, place={Berlin}, }, date={1982},
      pages={10--22}, }

    \bib{BrPezz1}{article}{ author={Bravi, P.}, author={Pezzini, G.},
      title={A constructive approach to the classification of
        wonderful varieties}, journal={Preprint}, pages={34 pages},
      date={2011},
      arxiv={1103.0380} }

    \bib{BrPezz2}{article}{ author={Bravi, P.}, author={Pezzini, G.},
      title={Primitive wonderful varieties}, journal={Preprint},
      pages={16 pages}, date={2011},
      arxiv={1106.3187} }

    \bib{BrPezz3}{article}{ author={Bravi, P.}, author={Pezzini, G.},
      title={The spherical systems of the wonderful reductive
        subgroups}, journal={Preprint}, pages={12 pages}, date={2011},
      arxiv={1109.6777} }

    \bib{BriES}{article}{ author={Brion, M.}, title={Vers une
        généralisation des espaces symétriques}, journal={J. Algebra},
      volume={134}, pages={115--143}, date={1990} }

    \bib{CuFou}{article}{ author={Cupit-Foutou, S.}, title={Wonderful
        varieties: a geometrical realization}, journal={Preprint},
      pages={34 pages}, date={2009},
      arxiv={0907.2852} }

    \bib{EGA4}{article}{ label={EGA4} author={Grothendieck, A.},
      title={Éléments de géométrie algébrique. IV. Étude locale des
        schémas et des morphismes de schémas. II},
      journal={Inst. Hautes Études Sci. Publ. Math.}, number={24},
      date={1965}, pages={231} }

    \bib{HaLau}{article}{ author={Haboush, W.}, author={Lauritzen,
        N.}, title={Varieties of unseparated flags}, conference={
        title={Linear algebraic groups and their representations (Los
          Angeles, CA, 1992)}, }, book={ series={Contemp. Math.},
        volume={153}, publisher={Amer. Math. Soc.}, place={Providence,
          RI}, }, date={1993}, pages={35--57}, }

    \bib{Jou}{book}{ author={Jouanolou, J.-P.}, title={Théorèmes de
        Bertini et applications}, series={Progress in Mathematics},
      volume={42}, publisher={Birkhäuser Boston Inc.}, place={Boston,
        MA}, date={1983}, pages={ii+127}, }

    \bib{Hyd}{article}{ author={Knop, F.}, title={The Luna-Vust theory
        of spherical embeddings}, conference={ title={Proceedings of
          the Hyderabad Conference on Algebraic Groups (Hyderabad,
          1989)}, }, book={ publisher={Manoj Prakashan},
        place={Madras}, }, date={1991}, pages={225--249}, }

    \bib{InvBew}{article}{ author={Knop, F.}, title={Über Bewertungen,
        welche unter einer reduktiven Gruppe invariant sind},
      journal={Math. Ann.}, volume={295}, pages={333--363},
      date={1993} }

    \bib{BorSub}{article}{ author={Knop, F.}, title={On the set of
        orbits for a Borel subgroup}, journal={Comment. Math. Helv.},
      volume={70}, pages={285--309}, date={1995} }

    \bib{KnRk1}{article}{ author={Knop, F.}, title={Homogeneous
        varieties for semisimple groups of rank one},
      journal={Compositio Math.}, volume={98}, pages={77--89},
      date={1995} }

    \bib{SpRoot}{article}{ author={Knop, F.}, title={Spherical roots
        of spherical varieties}, journal={Preprint}, pages={19 pages},
      date={2013},
      arxiv={1303.2466} }

    \bib{Losev}{article}{ author={Losev, I.}, title={Uniqueness
        property for spherical homogeneous spaces}, journal={Duke
        Math. J.}, volume={147}, pages={315--343}, date={2009},
      arxiv={0703543} }

    \bib{LuGC}{article}{ author={Luna, D.}, title={Grosses cellules
        pour les variétés sphériques}, conference={ title={Algebraic
          groups and Lie groups}, }, book={
        series={Austral. Math. Soc. Lect. Ser.}, volume={9},
        publisher={Cambridge Univ. Press}, place={Cambridge}, },
      date={1997}, pages={267--280}, }

    \bib{LuTypA}{article}{ author={Luna, D.}, title={Variétés
        sphériques de type A}, journal={Publ. Math. Inst. Hautes
        Études Sci.}, volume={94}, pages={161–226}, date={2001} }

    \bib{LuVu}{article}{ author={Luna, D.}, author={Vust, Th.},
      title={Plongements d'espaces homogènes},
      journal={Comment. Math. Helv.}, volume={58}, date={1983},
      pages={186--245}, }

    \bib{Schalke}{article}{ author={Schalke, B.}, title={Sphärische
        Einbettungen in positiver Charakteristik},
      journal={Diplomarbeit (Universität Erlangen)}, pages={39 pages},
      date={2011} }

    \bib{Sum}{article}{ author={Sumihiro, H.}, title={Equivariant
        completion}, journal={J. Math. Kyoto Univ.}, volume={14},
      pages={1--28}, date={1974} }

    \bib{Wass}{article}{ author={Wasserman, B.}, title={Wonderful
        varieties of rank two}, journal={Transform. Groups},
      volume={1}, pages={375--403}, date={1996} }

    \bib{Wenz0}{article}{ author={Wenzel, Ch.}, title={Classification
        of all parabolic subgroup-schemes of a reductive linear
        algebraic group over an algebraically closed field},
      journal={Trans. Amer. Math. Soc.}, volume={337}, date={1993},
      pages={211--218}, }

    \bib{Wenz}{article}{ author={Wenzel, Ch.}, title={On the structure
        of nonreduced parabolic subgroup-schemes}, conference={
        title={}, address={University Park, PA}, date={1991}, },
      book={ series={Proc. Sympos. Pure Math.}, volume={56},
        publisher={Amer. Math. Soc.}, place={Providence, RI}, },
      date={1994}, pages={291--297}, }

  \end{biblist}
\end{bibdiv}
\end{document}